\documentclass[11pt]{article}
\usepackage{graphicx}
\graphicspath{ {./sdraft/imggit/} }

\usepackage[T1]{fontenc}
\usepackage[a4paper]{geometry}
\usepackage{amsmath, amsfonts, amsthm, amssymb}
\usepackage{xcolor}
\usepackage{mathtools}
\usepackage{bbm}
\usepackage{hyperref}
\usepackage{nicefrac}
\usepackage{algorithm}
\usepackage{algpseudocode}
\usepackage{todonotes}
\usepackage{enumitem}
\usepackage{units}
\usepackage{lipsum}
\usepackage{subcaption}

\newcommand{\hhmu}{h,\hat{\mu}}
\newcommand{\Hcal}{\mathcal{H}}

\newcommand{\uFOM}{u_{\mathrm{FOM}}}
\newcommand{\bu}{\mathbf{b}}
\newcommand{\Th}{\mathcal{T}_{h}}

\newcommand{\aSUPG}{a_{\mathrm{SUPG}}}
\newcommand{\dt}{\triangle t}
\newcommand{\uhr}[1][]{u^{#1}_{h,\hat{\mu}}}

\newcommand{\lrX}[1]{\langle #1 \rangle_{X}}
\newcommand{\trnorm}[1]{{\left\vert\kern-0.25ex\left\vert\kern-0.25ex\left\vert #1 
\right\vert\kern-0.25ex\right\vert\kern-0.25ex\right\vert}}

\newcommand{\Lthm}{L^2_{\hat{\mu}}(\hat{\Omega})}

\newcommand{\hmu}{\hat{\mu}}

\newtheorem{theorem}{Theorem}[section]
\newtheorem{proposition}{Proposition}[section]
\newtheorem{lemma}{Lemma}[section]
\newtheorem{remark}{Remark}

\newtheorem*{tancond}{Skewing conditions}
\newtheorem*{DOeqs}{Dual DO evolution equations}
\newtheorem*{dlrprob}{Dynamical Low Rank (DLR) Problem}
\newtheorem*{fomprob}{Full Order Model (FOM) Problem}
\newtheorem*{adrfomprob}{Advection-diffusion-reaction FOM Problem}
\newtheorem*{sfomprob}{Generalised Petrov-Galerkin FOM (PG-FOM) Problem}
\newtheorem*{sdlrprob}{Generalised Petrov-Galerkin DLR (PG-DLR) Problem}

\newtheorem{alg}{Algorithm}

\newcommand{\modt}[1]{{\color{black} #1}}

\title{Petrov-Galerkin Dynamical Low Rank Approximation: SUPG stabilisation of advection-dominated problems}
\author{Fabio Nobile, Thomas Trigo Trindade}

\begin{document}

\maketitle

\begin{abstract}

We propose a novel framework of generalised Petrov-Galerkin Dynamical Low Rank Approximations (DLR) in the context of random PDEs. 
It builds on the standard Dynamical Low Rank Approximations in their Dynamically Orthogonal formulation. 
It allows to seamlessly build-in many standard and well-studied stabilisation techniques that can be framed as either generalised Galerkin methods, or Petrov-Galerkin methods. 
The framework is subsequently applied to the case of Streamine Upwind/Petrov Galerkin (SUPG) stabilisation of advection-dominated problems with small stochastic perturbations of the transport field.
The norm-stability properties of two time discretisations are analysed. 
Numerical experiments confirm that the stabilising properties of the SUPG method naturally carry over to the DLR framework. 

\end{abstract}



\section{Introduction}

The efficient and accurate simulation of random or parametric time-dependent Partial Differential Equations (PDEs) is of high interest in the broad context of Uncertainty Quantification.
Repeated solves of the problem for different values of the (possibly random) parameters are needed in many contexts such as parametric sensitivity, Bayesian inference on parameters, uncertainty quantification of output quantities of interest or parameter optimisation in systems design. 
When the underlying PDE is discretised at high accuracy on a fine grid, the overall cost of these tasks becomes prohibitive and calls for Reduced Order Models to make it tractable. 

Fixed Reduced Basis methods such as the popular Proper Orthogonal Decomposition (POD) and its variations~\cite{audena09, buimuwi03, chadbrso99, voka07, voku01} are not well-suited for time-dependent problems if the global-in-time solutions have slowly decaying Kolomogorov $n$-width.
This occurs even if the solution is well-approximated at each time by a small subspace, but that subspace changes significantly over time. 
A prime example of this is given by time-dependent advection-dominated problems with the translation of coherent features.   
Some work has been put in tackling this issue and often involves the use of non-linear transformation of coordinates. 
One prominent example is given by the shifted POD method~\cite{rescseme18}, alongside other types of transported snapshot methods~\cite{naba19, nobaro23} or the \textit{freezing} method~\cite{ohra13, roke03}, which relies on a suitable Lie group acting on a frozen solution. 
In this paper, we consider an alternative approach, namely the Dynamical Low Rank Approximation method~\cite{kolu07}.
More specifically, we consider the Dynamically Orthogonal (DO) formulation~\cite{sale09}, originally developed in the context of random PDEs, which was subsequently understood to be the application of the DLR method to that setting~\cite{munozh15}. 
It consists in seeking a (truncated) Karhunen-Lo\`eve-like approximation of the true solution 
\begin{equation*}
u_{\mathrm{true}}(t,x,\omega) \approx \sum_{i=1}^R {U_i(t,x) Y_i(t, \omega)} \eqqcolon u_{\mathrm{DLR}}(t,x,\omega), 
\end{equation*}
where $x$ denotes the physical variable, $t$ denotes time and $\omega$ the (possibly random) parameters. 
The defining feature of this approach is that the factors $U_i$ and $Y_i$ of the approximate solution dynamically evolve in time -- in other words, and contrary to fixed-basis methods, the low-rank subspace $\mathrm{span}\{U_1, \ldots, U_R\}$ in which the solution is sought adapts in time to accommodate the onset of new features that may appear over time. 

Another practical issue that arises with advection-domiated problems is that of stabilising numerical solutions.
The choice of discretisation can result in numerical instabilities or artefacts, and for practical purposes, cancelling or alleviating those numerical instabilities is of obvious interest. 
The simulation of advection-dominated fluids (such as the Oseen or Navier-Stokes problems) provide a good example in this context. 
It is well-known that naively applying a standard Finite Element Method may lead the solution to display unphysical oscillations, especially so when the solution involves sharp gradients or boundary layers. This has lead to a rich and increasingly mature theory of stabilisation techniques; for advection-dominated problems, see e.g.~\cite{arn82, bedego19, brhu82, bufeha06,bubaga22} and references therein.

In order to build efficient and stabilised numerical approximations of random time-dependent problems, this paper introduces a new generalised Petrov-Galerkin Dynamical Low Rank (PG-DLR) framework as an extension to standard DLR. 
Stabilisation methods that can be phrased as \textit{generalised} Galerkin methods are, as modifications to the dynamics (``right-hand-side''), in essence always possible. 
Such stabilisations have in fact already been recently used in the DLR literature, see e.g. the discretisation of the pure-advection PDE using discontinuous finite elements with upwinding in~\cite{cosc23}, which can be understood as a stabilisation technique. 
The \textit{Petrov}-Galerkin aspect comes with a few caveats. 
Many non-trivial questions must be addressed, among which: what are suitable testing spaces, different from the trial space, that may be considered, given that the DLR framework is primarily defined by testing against a specific time-evolving test space (the tangent space) ? 
under what conditions can the modes be updated efficiently?
are the stabilisation schemes guaranteed to have the same effect as their full order model counterparts?
is it possible to use DLR time-stepping schemes robust with respect to the smallest singular values to update the system?
are standard error estimates recoverable? etc.

%

\modt{
  In this work, we provide positive answers to all of the questions outlined above save for the last point -- that is the topic of a separate work~\cite{notr24}, in which we perform an error analysis of the Streamline Upwind/Petrov Galerkin (SUPG)-stabilised DLR.
Our framework builds the Petrov-Galerkin stabilisation technique into the DLR method and applies the stabilisation on-the-fly to the dynamically evolving bases. 
The interplay between the stabilisation method and the DLR part is not trivial and a major focus of this work to propose a formulation which allows for an efficient implementation of the stabilised DLR scheme by leveraging the structure of the stabilising operators. 
In particular, stabilisation techniques that modify solely one of the spaces (physical or probabilistic) fall in this category, and, in this respect, various Finite Element Method-based stabilisation methods can be seamlessly integrated into the DLR formulation, making the framework relevant for practical applications. 
This work inscribes itself in a growing body of literature tackling the stabilisation of Reduced Order Models (ROMs), many of which focus on the stabilised Reduced Basis (RB) simulation of fluids~\cite{de12,bagiali20,bagiali23}. 
In particular, for the SUPG-stabilised advection-dominated problems, see~\cite{tobaro18, paro14,giiljowe15}. 
The authors in~\cite{paro14} argue that the correct approach seems to be applying the stabilisation both in the offline and online phase -- a feature that, \textit{mutatis mutandis}, our PG-DLR framework automatically integrates. 
The key point in which our method differs from the usual (stabilised) fixed RB paradigm is that it dynamically stabilises the time-dependent bases, hence eschewing the cost of the offline phase while enjoying the advantges of the stabilisation method. 
Closer to our context, we also mention the stabilisation of DO approximations in~\cite{fele18} by Shapiro filters to smooth out the arising oscillations in a post-processing fashion after each time step.}

Upon establishing our framework for PG-DLR, we extend to our context the time-integration schemes proposed in~\cite{kavino21} and study their stability properties.
These schemes have the major advantage of making it possible to recover a dicrete variational formulation, a useful tool when performing the norm-stability and error analyses of the time-integrator. 
We apply the framework to the case of advection-dominated problems using the SUPG method, yielding an SUPG-stabilised Dynamical Low Rank Approximation. 
The norm-stability of two schemes (an implicit Euler scheme and a semi-implicit one) are established in this work; \modt{the error analysis is performed in~\cite{notr24}}.

The paper is structured as follows: in Section~\ref{sec:not}, we formalise the problem and review the concepts of generalised Petrov-Galerkin problems, SUPG stabilisation and the Dynamical Low Rank/Dynamically Orthogonal framework. 
Section~\ref{sec:pg-dlr} introduces the generalised Petrov-Galerkin Dynamical Low Rank (PG-DLR) framework in a generic setting, and discusses the scope and limitations of the method. 
In Section~\ref{sec:supg-dlr}, we particularise the framework to the case of SUPG-stabilisation for advection-dominated problems with deterministic advection or mean-dominated advection with small stochastic fluctuations, and discuss two possible time-discretisations (implicit Euler and semi-implicit) and detail their norm-stability properties. The main conclusion is that, as in~\cite{kavino21}, PG-DLR solutions are expected to display in essence the same properties as their full order model counterparts. 
Finally, Section~\ref{sec:num-exp} illustrates the stabilising effect of SUPG-DLR through numerical examples.


\section{Problem statement} \label{sec:not}

Let $(\Omega, \mathcal{S}, \mu)$ be a complete probability space and let $L^2_{\mu}(\Omega) = \{ f : \Omega \rightarrow  \mathbb{R}, \displaystyle\int_{\Omega} |f|^2 \mathrm{d}\mu < \infty \}$ be the space of square-integrable random variables. Let $D \subset \mathbb{R}^d$ be an open connected domain (for some $d \geq 1$) and let 
$X$ and $V$ be separable Hilbert spaces on $D$ with corresponding scalar product $\langle\cdot, \cdot \rangle_{\star}$ and induced norm $\|u\|^2_{\star} = \langle u, u \rangle_{\star}$ (with $\star = X,V$, resp.) and such that $(V,X,V^{\prime})$ form a Gelfand triple. 
We denote by $L^2_{\mu}(\Omega, \star)$ the Bochner spaces defined as
\[
	L^2_{\mu}(\Omega, \star) = \{ v : \Omega \rightarrow \star, v \; \text{strongly measurable}, \int_{\Omega} \|v\|^2_{\star} \; \mathrm{d}\mu < \infty \},
\]
with scalar product given by 
\[
	(u,v)_{\star,L^2_{\mu}(\Omega)} = \int_{\Omega} \langle u(\omega), v(\omega) \rangle_{\star} \mathrm{d}\mu(\omega).
\]
We consider the following abstract random time-dependent problem
\begin{align} 
	\dot{u}(t, x, \omega) + Lu(t,x,\omega) &= f(t,x,\omega) & \text{a.e. } t \in (0,T], x \in D, \mu\text{-a.e. } \omega \in \Omega \label{eqn:ab-time-dep} \\
	u(0, x, \omega) &= u_0(x, \omega) &  x \in D, \mu\text{-a.e. }\omega \in \Omega \\
	\mathcal{B}(u)(t, \sigma, \omega) &= 0 & \text{a.e. }t \in (0, T],  \sigma \in \partial D, \mu\text{-a.e. }\omega \in \Omega
\end{align}
where the operator $L$, the initial condition $u_0$, the source term $f$, and the boundary conditions $\mathcal{B}(u)$ may be stochastic. 
In what follows, it is assumed that $L : L^2_{\mu}(\Omega, V) \rightarrow L^2_{\mu}(\Omega, V^{\prime})$ is time-independent and linear \modt{for ease of presentation. The PG-DLR framework also accomodates more general, non-linear operators such as Navier-Stokes.}
The operator $L$ induces a bilinear form $\mathcal{A}(u,v) = (Lu, v)_{V^{\prime}V, L^2_{\mu}(\Omega)}$ on $L^2_{\mu}(\Omega, V) \times L^2_{\mu}(\Omega, V)$, where $(\cdot, \cdot)_{V^{\prime}V,L^2_{\mu}(\Omega)}$ denotes the duality pairing between $L^2_{\mu}(\Omega, V^{\prime})$ and $L^2_{\mu}(\Omega, V)$. 
In this work, we assume $f \in L^2_{\mu}(\Omega, X)$; it induces a linear functional $(f,v)_{X,L^2_{\mu}(\Omega)} = \mathcal{F}(v)$ on $L^2_{\mu}(\Omega, X)$. 
The arguments that follow can however be extended to more general forcing terms $f \in L^2_{\mu}(\Omega, V')$. 
For simplicity of exposition, we restrict ourselves to the case of (deterministic) homogeneous Dirichlet boundary conditions, and assume that the functions in $V$ satisfy those boundary conditions. 
The problem in weak form reads

\noindent \textit{Find} $u \in L^2((0,T), L^2_{\mu}(\Omega, V))$ \textit{with} $\dot{u} \in L^2((0,T), L^2_{\mu}(\Omega, V^{\prime}))$ \textit{ s.t. }
\begin{equation} \label{eqn:weak-form-gen}
	\left\{
		\begin{aligned}
			& (\dot{u}, v)_{V'V,L^2_{\mu}(\Omega)} + \mathcal{A}(u, v) = \mathcal{F}(v), \quad \mathrm{a.e.}\;t \in (0,T], \forall v \in L^2_{\mu}(\Omega, V) 
 \\
			& u_{|t=0} = u_0 \in L_{\mu}^2(\Omega, X),
		\end{aligned}
	\right.
	\end{equation}
It is assumed that sufficient conditions are met to ensure the problem is well-posed, e.g. coercivity and continuity of $\mathcal{A}$ with respect to the $L^2_{\mu}(\Omega, V)$-norm, and continuity of $\mathcal{F}(v)$ w.r.t. to the $L^ 2_{\mu}(\Omega, X)$-norm.

\subsection{Full Order Model (FOM) discretisation}

In this work, we discretise the physical space by the Finite Element Method, using a quasi-uniform mesh $\Th \subset D$ where $h$ denotes the mesh size.
The Finite Element space is denoted $V_h \subset V$ with $\mathrm{dim}(V_h) \eqqcolon N_h$.  
The probability space $(\Omega, \mathcal{S}, \mu)$ is discretised using a generic collocation method (e.g., Monte Carlo, Quasi Monte Carlo, Stochastic Collocation or suitable quadrature nodes, etc) to sample a set of $N_C$ points $\hat{\Omega} \coloneqq \{\omega_i\}_{i=1}^{N_C} \subset \Omega$, and replacing $\mu$ with an empirical measure $\hmu \approx \mu$ such that $\hat{\mu} \coloneqq \sum_{i=1}^{N_C} m_i \delta_{\omega_i}$ where $\{m_i\}_{i=1}^{N_C}$  are \emph{positive} weights which sum up to $1$. 
The new scalar product on $L^2_{\hat{\mu}}(\hat{\Omega})$ is given by the empirical mean, 
\[
	\langle f, g \rangle_{\hat{\mu}} = \mathbb{E}_{\hat{\mu}}[f g] = \sum_{i=1}^{N_C} m_i f(\omega_i) g(\omega_i). 
\]
We define $\mathbbm{1} \in L^2_{\hat{\mu}}(\hat{\Omega})$ the constant random variable equal to $1$.
The \textit{(empirical) mean} of $u \in V_h \otimes \Lthm$ is denoted as
\[
	\bar{u} \coloneqq \langle u, \mathbbm{1} \rangle_{\hat{\mu}} = \mathbb{E}_{\hat{\mu}}[u],
\]
and the \textit{zero-mean part} of $u \in V_h \otimes \Lthm$ as 
\[
	u^{\star} \coloneqq u - \bar{u}.
\]
With suitable modifications, the framework described below can be extended to Stochastic Galerkin methods~\cite{ghsp91} (e.g., combined with generalised Polynomial Chaos expansions~\cite{doka02}).

It is also assumed that an inverse inequality holds in $V_h$ between the norms $\| \cdot \|_V$ and $\| \cdot \|_X$ namely for some $k \geq 1$ and $C_I >0$, $\|u_h\|_{V} \leq C_I h^{-k} \|u_h\|_X$  for any $u_h \in V_h$. This inequality naturally extends to elements in $V_h \otimes L^2_{\hat{\mu}}$ with the same $C_I$ and $k$:
\begin{equation} \label{eqn:inv-equal}
	\|u_{\hhmu}\|_{V,\Lthm} \leq \frac{C_I}{h^k} \|u_{\hhmu}\|_{X,\Lthm}, \quad \forall u_{\hhmu} \in V_h \otimes \Lthm,
\end{equation}
as the inequality holds pointwise in $\omega$ and the weights are positive.  
Hereafter, we will use the shorthand notation $\| \cdot \|^2 = (\cdot, \cdot) \equiv (\cdot, \cdot)_{X,L^2_{\hat{\mu}}(\hat{\Omega})}$ for brevity.

%
The approximation of the (weak) solution to~\eqref{eqn:ab-time-dep} is then given by
\begin{fomprob} \noindent\textit{Find } $u_{\mathrm{FOM}} : (0,T] \rightarrow  V_h \otimes L^2_{\hat{\mu}}(\hat{\Omega})$ \textit{s.t.} 
\begin{equation} \label{eqn:weak-form}
	(\dot{u}_{\mathrm{FOM}}, v_{\hhmu})_{X,\Lthm} + \mathcal{A}(u_{\mathrm{FOM}}, v_{\hhmu}) = \mathcal{F}(v_{\hhmu}), \qquad \forall v_{\hhmu} \in V_h \otimes L^2_{\hat{\mu}}(\hat{\Omega}),\; \mathrm{a.e.}\; t \in (0,T] .
\end{equation}
\end{fomprob}
\noindent The high dimensionality of the system $\mathrm{dim}(V_h \otimes L^2_{\hat{\mu}}(\hat{\Omega})) = N_C N_h$ may become intractable in practice and calls for the use of ROMs. 
Furthermore, numerical artefacts may appear in the numerical solution, e.g. for transport-dominated problems, which need to be controlled for practical purposes. 

\subsection{Generalised Petrov-Galerkin framework}

The \textit{Petrov-Galerkin framework} describes variational problems in which the trial space is different from the test space. 
The choice of the test space and its norm prove to have important consequences on the well-conditioning of the problem, which is then often directly reflected by the quality of the numerical approximation. 
Consequently, much attention has been devoted to the investigation of \textit{optimal} test spaces and \textit{optimal} test functions, particularly for advection-diffusion problems (see e.g.,~\cite{sa05,codawe12,dahusc12, dego11} and references therein).
This analysis can be carried out in the infinite-dimensional context before discretisation.

Conversely, the \textit{generalised Galerkin} framework is applied post-discretisation, and consists in replacing the operators $\mathcal{A}$, $\mathcal{F}$ by discrete approximations $\mathcal{A}_h$, $\mathcal{F}_h$, with the aim of stabilising the problem in some fashion. 
We focus on a framework allowing the combination of both approaches, hereafter labelled \textit{generalised Petrov-Galerkin} framework. 
We consider the case where the modified test space is obtained through a linear operator $\Hcal : V_{h} \otimes \Lthm \rightarrow X_h \otimes \Lthm \subset L^2_{\hat{\mu}}(\hat{\Omega}, X)$ acting on the original test space with values in $X_h$, another finite element space contained in $X$ but not necessarily in $V$. 
Furthermore, in this work we consider the case $\mathcal{F}_h(v_{\hhmu}) = \mathcal{F}(\Hcal v_{\hhmu})$, which is valid since $f \in L^2_{\hat{\mu}}(\hat{\Omega}, X)$.
The numerical approximation of~\eqref{eqn:ab-time-dep} is thus obtained by solving the (formal) problem:

\begin{sfomprob}
\noindent \textit{Find } $u_{\mathrm{PGFOM}} : (0,T] \rightarrow V_h \otimes \Lthm$  \textit{ s.t.}
\begin{multline} \label{eqn:weak-form-pg}
	(\dot{u}_{\mathrm{PGFOM}}, \Hcal v_{\hhmu})_{X,\Lthm} + \mathcal{A}_h( u_{\mathrm{PGFOM}}, v_{\hhmu}) = \mathcal{F}(\Hcal v_{\hhmu}), 
	\\ 
	\forall v_{\hhmu} \in V_h \otimes L^2_{\hat{\mu}}(\hat{\Omega}), \; \mathrm{a.e.}\; t \in (0,T].
\end{multline}
\end{sfomprob}
\noindent This problem need not necessarily be strongly consistent with respect to~\eqref{eqn:ab-time-dep}. 
However, if the starting point is a skewing of the test space by~$\Hcal$, strong consistency may be recovered by ensuring that $\mathcal{A}_h(u, v_{\hhmu}) = \mathcal{A}(u, \Hcal v_{\hhmu})$ for $u$ the true solution.
An obvious condition $\Hcal$ must verify for Problem~\eqref{eqn:weak-form-pg} to be well-posed is $\mathrm{Im}(\Hcal) \simeq V_h \otimes \Lthm$. 
The formulation~\eqref{eqn:weak-form-pg} can also accommodate the case $\Hcal = \mathrm{Id}$ and modified $\mathcal{A}$ and $\mathcal{F}$ to account e.g. for quadrature formulas or other approximations, leading straightforwardly to a generalised Galerkin problem.

%

The bilinear form $\mathcal{A}_h$ is defined on $L^2_{\hat{\mu}}(\hat{\Omega}, V_h) \times L^2_{\hat{\mu}}(\hat{\Omega}, V_h)$, and therefore by the Riesz representation theorem the functional $\mathcal{A}_h(u, \cdot) : L^2_{\hat{\mu}}(\hat{\Omega}, V_h) \rightarrow \mathbb{R}$ can be recast as $\mathcal{A}_h(u_{\hhmu}, \cdot) = (L_h u_{\hhmu}, \cdot)_{V_h^{\prime}V_h,\Lthm}$ in $L^2_{\hat{\mu}}(\hat{\Omega}, V_h')$ with $L_h : L^2_{\hat{\mu}}(\hat{\Omega}, V_h) \rightarrow L^2_{\hat{\mu}}(\hat{\Omega}, V_h^{\prime})$. 
The problem may be rewritten as
\begin{multline} \label{eqn:weak-form-pg-2}
	(\dot{u}_{\mathrm{PGFOM}}, \Hcal v_{\hhmu})_{X,\Lthm} + (L_h u_{\mathrm{PGFOM}}, v_{\hhmu})_{V^{\prime}_h V_h,\Lthm} = (f, \Hcal v_{\hhmu})_{X,\Lthm}, 
	\\
	\forall v_{\hhmu} \in V_h \otimes L^2_{\hat{\mu}}(\hat{\Omega}), \; \mathrm{a.e.}\; t \in (0,T]
\end{multline}
Numerous finite element stabilisation frameworks fall within this formulation, such as the interior penalty method~\cite{arn82}, or, in the case of advection-dominated problems, residual-based methods such as Galerkin Least-Squares~\cite{bedego19}, the Douglas-Wang method or the SUPG method~\cite{brhu82}, which will be described in more detail in the next section.

\subsubsection{Streamline Upwind Petrov Galerkin (SUPG) stabilisation}

Many physical processes are governed by the transient advection-diffusion-reaction equation,  
\begin{equation} \label{eqn:adv-diff-reac}
	\begin{aligned} 	\partial_t u  - \varepsilon \Delta u + \mathbf{b} \cdot \nabla u + c u &= f, && \text{a.e. } t \in (0,T], x \in D,\mu\text{-a.e. } \omega \in \Omega \\
	u(0, x, \omega) &= u_0(x, \omega), && x \in D, \omega \in \Omega, \\
	u(t, \sigma, \omega) &= 0, && t \in (0,T], \sigma \in \partial D, \omega \in \Omega.
\end{aligned}
\end{equation}

Such equations are also a prototype for more complex models such as the incompressible or compressible Navier-Stokes equations.
Concerning the coefficients in~\eqref{eqn:adv-diff-reac}, we assume that
\begin{equation} \label{eqn:varepscoer}
C_E \hat{\varepsilon} \geq \varepsilon(\omega) \geq \hat{\varepsilon} > 0 \quad \forall \omega \in \Omega.
\end{equation}
The diffusion coefficient $\varepsilon$ could also depend on $x$ as long as it verifies the same uniform upper and lower bound. 
In this case, it would be more natural to consider as diffusion term $- \mathrm{div}(\varepsilon \nabla u )$ rather than $- \varepsilon \Delta u$ in~\eqref{eqn:adv-diff-reac}.

We denote the quantity
\begin{equation} 
	\tilde{\mu}(x, \omega) \coloneq c(x, \omega) - \frac{|c(x, \omega)|}{2} - \frac{1}{2}\mathrm{div}\, \bu(x, \omega), \; \mathrm{for}\; x\in D, \omega \in \Omega,
\end{equation}
and define $\displaystyle \tilde{\mu}_0 = \inf_{x \in D, \omega \in \Omega} \tilde{\mu}(x,\omega)$ and $\nu = - \min \left\{ \tilde{\mu}_0, 0\right\}$, whence it follows
\begin{equation}
	\mu(x,\omega) \coloneqq \tilde{\mu}(x, \omega) + \nu \geq 0, \; x \in D, \omega \in \Omega.
\end{equation}
Finally, we denote $\displaystyle \mu_0 \coloneqq \inf_{x \in D, \omega \in \Omega} \mu(x,\omega)$. The condition~$\displaystyle \inf_{x \in D, \omega \in \Omega}(c - \nicefrac{1}{2} \mathrm{div}\, \bu) > 0$ is often assumed in the SUPG literature~\cite{jovo11, rohast08}. 
\modt{This condition does not allow to directly treat certain cases of interest, such as zero-divergence advection field and $c = 0$, and requires to perform some suitable change of variables first}. Therefore we avoid using it in the analysis that follows and only require that $\|c\|_{L^{\infty}(\hat{\Omega} \times D)}$, $\|\bu\|_{L^{\infty}(\hat{\Omega} \times D)}$, $\|\mathrm{div}\, \bu \|_{L^{\infty}(\hat{\Omega} \times D)} < \infty$.

Taking the weak form and discretising by the Galerkin-collocation method described above, the standard full order model problem~\eqref{eqn:adv-diff-reac} reads

\noindent\textit{Find } $u_{\mathrm{FOM}} : (0,T] \rightarrow  V_h \otimes L^2_{\hat{\mu}}(\hat{\Omega})$ \textit{s.t.} 
\begin{multline} \label{eqn:fom-varf}
	(\dot{u}_{\mathrm{FOM}}, v_{\hhmu}) + (\varepsilon \nabla u_{\mathrm{FOM}}, \nabla v_{\hhmu})
	+ (\bu \cdot \nabla \uFOM, v_{\hhmu}) + (c\uFOM, v_{\hhmu})
	= (f, v_{\hhmu}), \\ \qquad \forall v_{\hhmu} \in V_h \otimes L^2_{\hat{\mu}}(\hat{\Omega}), \mathrm{a.e.}\; t \in (0,T].
\end{multline}

The numerical simulation of these processes using the finite elements method on some triangulation $\Th \subset D$ turns out to be challenging, as standard FEM solutions are known to develop spurious oscillations when the dynamics is \textit{advection-dominated}, which happens when the \textit{local Péclet number} verifies
\begin{equation} \label{eqn:local-peclet}
	\mathbb{P}\mathrm{e}_K(x,\omega) = \frac{| \bu(x, \omega) | h_K}{2 \varepsilon(\omega)} > 1,  \quad \text{for some } K \in \Th, \forall x \in K, \forall \omega \in \hat{\Omega},
\end{equation}
$h_K$ denoting the diameter of the element $K \in \Th$. 

We particularise the general Petrov-Galerkin framework~\eqref{eqn:weak-form-pg} to the SUPG-stabilised finite element approximation applied to~\eqref{eqn:adv-diff-reac}.
In this case, $V = H^1_0(D)$, $X = L^2(D)$. Assuming the dependence of $\bu$ on $x$ is polynomial, we set $V_h = \mathbb{P}^C_k(\Th)$ (Continuous Galerkin finite elements of degree $k$) and $X_h = \mathbb{P}^{DG}_{k'}(\Th)$ (Discontinuous Galerkin finite elements of degree $k'$, to take into account the additional polynomial terms stemming from $\bu \cdot \nabla V_h$).
In this Finite Element setting, the inverse inequality~\eqref{eqn:inv-equal} holds with $k=1$, i.e., $\|u_{\hhmu}\|_{V,\Lthm}~\leq~C_I h^{-1}\|u_{\hhmu}\|_{X,\Lthm} $.

The SUPG method is a popular method with an extensive literature~\cite{brhu82, kn08, bu10} dedicated to it that tackles this problem. The SUPG bilinear form is defined as~\begin{align} \label{eqn:asupgref}
	\aSUPG(u, v) = (\varepsilon \nabla u, \nabla v) &+ (\bu \cdot \nabla u, v) + (cu, v) \nonumber
	\\ 
																									&+ \sum_{K \in \Th} \delta_K (- \varepsilon \Delta u + \bu \cdot \nabla u + cu, \bu \cdot \nabla v)_{K,\Lthm}, 
\end{align}
where $(u,v)_{K,\Lthm} = \sum_{i=1}^{N_C} m_i (u(\cdot, \omega_i),v(\omega_i))_K$.
The SUPG method consists in seeking an approximation to the original problem by solving
\begin{adrfomprob}
\noindent\textit{Find } $u_{\mathrm{PGFOM}} : (0,T] \rightarrow  V_h \otimes L^2_{\hat{\mu}}(\hat{\Omega})$ \textit{s.t.} 
\begin{multline} \label{eqn:supgsfom}
	(\dot{u}_{\mathrm{PGFOM}}, v_{\hhmu}) + \sum_{K \in \Th} \delta_K ( \dot{u}_{\mathrm{PGFOM}}  ,  \bu \cdot \nabla v_{\hhmu})_{K,\Lthm} + \aSUPG(\uFOM, v_{\hhmu}) \\ 
	= (f, v_{\hhmu})  + \sum_{K \in \Th} \delta_K (f, \bu \cdot \nabla v_{\hhmu})_{K,\Lthm} \qquad \forall v_{\hhmu} \in V_h \otimes L^2_{\hat{\mu}}(\hat{\Omega}), \; \mathrm{a.e.}\; t \in (0,T].
\end{multline}
\end{adrfomprob}
In the case where $\mathrm{div} \, \bu = 0$, this can be understood as adding viscosity in the direction of the transport. We refer to~\cite{quva08} for a more in-depth discussion of this topic. 

The SUPG method straightforwardly falls into the class of problems of the form~\eqref{eqn:weak-form-pg}, where $\mathcal{A}_h(u,v) = \aSUPG(u,v)$, $\mathcal{F}_h(v) = (f, v)+  \sum_{K \in \Th} \delta_K (f, \bu \cdot \nabla v)= (f, \Hcal v)_{K,\Lthm} $, $\Hcal$ being given by   
\begin{equation} \label{eqn:supgHcal}
	\Hcal v = v + \sum_{K \in \Th} \delta_K \chi_{K} \bu \cdot \nabla v 
\end{equation}
where $\chi_K$ is the indicator function on $K \in \Th$. 

Furthermore, the following norm (defined on each cell $K$ of the mesh) will prove to be useful: 
\begin{equation}
\trnorm{v}_K = \sup_{(x, \omega) \in K \times \hat{\Omega}} |v(x,\omega)|
\end{equation}

Finally, we establish the (weak) coercivity of the SUPG operator. 
The scope and statement of the lemma is slightly different from usual SUPG assumptions, but the proof (given in the appendix for completeness) follows very similar steps to~\cite[Lemma 3.25]{rohast08}.
\begin{lemma} \label{lem:coerasupg} Assuming~\eqref{eqn:varepscoer} and 
	\begin{equation} \label{eqn:deltaK-coerc}
		\delta_K \leq \left\{ \frac{1}{2 \trnorm{c}_K}, \frac{h_K^2}{2 d C_I^2 C^2_E \hat{\varepsilon} } \right\},
	\end{equation}
	it holds
	\begin{equation} \label{eqn:asupg-weak-coer}
		a_{\mathrm{SUPG}}(u_{\hhmu}, u_{\hhmu}) \geq \frac{1}{2} \|u_{\hhmu}\|^2_{\mathrm{SUPG}} - \nu \|u_{\hhmu}\|^2,
	\end{equation}
	where 
	\begin{equation}
		\|u_{\hhmu}\|_{\mathrm{SUPG}}^2 =  \hat{\varepsilon} \|\nabla u_{\hhmu}\|^2_{X,\Lthm} + \sum_{K \in \mathcal{T}_{\!h}} \delta_K \|\bu \cdot \nabla u_{\hhmu}\|^2_{K,\Lthm} + \| \mu^{\nicefrac{1}{2}} u_{\hhmu}\|^2_{X,\Lthm}
	\end{equation}
\end{lemma}
The second constraint on $\delta_K$ in~\eqref{eqn:deltaK-coerc} can be dropped when using $\mathbb{P}_1^{C}(\Th)$-finite elements. 
Condition~\eqref{eqn:deltaK-coerc} is similar to the one often found in the SUPG literature, which asks that $\displaystyle \delta_K \lesssim \inf_{x \in D, \omega \in \Omega}(c - \nicefrac{1}{2} \mathrm{div} \;\bu
) \|c\|_{L^{\infty}(K)}^{-2}$. Assuming zero-divergence of the advection field, the scalings on $\delta_K$ are of the same order. 


\subsection{Dynamical Low Rank (DLR) Approximations}

The Dynamical Low Rank (DLR) framework has been the subject of increasing attention and study in recent years. 
Originally developed in~\cite{kolu07} for time-dependent matrix equations, it has been independently investigated in the context of random PDEs in~\cite{sale09} under the name of Dynamically Orthogonal (DO) method, which is the formalism adopted in this work. 
Given a random problem $\dot{u} + \mathcal{L}u = f$, the DO/DLR method seeks an approximation of the true random field mimicking a truncated Karhunen-Loève expansion 
\begin{equation} \label{eqn:do-exp}	
	u_{\mathrm{true}}(t,x,\omega) \approx u_{\mathrm{DO}}(t,x,\omega) = \sum_{i=0}^R U_i(t,x) Y_i(t,\omega),
\end{equation}
where the \textit{physical} modes $\{U_i(t,x)\}_{i=0}^R$ and the \textit{stochastic} modes $\{Y_i(t, \omega)\}_{i=0}^R$ evolve in time to approximate the optimal low-rank representation of $u_{\mathrm{true}}(t,x,\omega)$.
The approximation belongs to the differential manifold $\tilde{\mathcal{M}}_R = \{ u = \sum_{i=0}^R U_i Y_i \;|\; U_i \in V_h, Y_i \in \Lthm \} \subset V_h \otimes L^2_{\hat{\mu}}$ at all times.
Following a common practice in the UQ community, we track the mean explicitly. 
To this end, we hereafter let $Y_0 \coloneqq \mathbbm{1}$. 
In order to avoid redundancy between the mean and other stochastic modes, the remaining stochastic modes are imposed to belong to the space of zero-mean stochastic modes, denoted $L^2_{\hat{\mu},0}(\hat{\Omega}) \coloneqq \{ Z  \in L^2_{\hat{\mu}}(\hat{\Omega}), \mathbb{E}_{\hat{\mu}}[Z] = 0 \}$.   
We define the \textit{differential manifold of $R$-rank functions} as 
\begin{multline}
	\mathcal{M}_{R} = \{u \in V_h \otimes L^2_{\hat{\mu}}(\hat{\Omega}) : u = U_0 Y_0 + \sum_{i=1}^R U_i Y_i, \text{ s.t. } \mathbb{E}_{\hat{\mu}}[Y_i Y_j] = \delta_{ij}, \{U_i\}_{i=1}^R \text{ lin. ind. } \\
 \text{and } \{U_i\}_{i=0}^R \in V_h, \{Y_i\}_{i=1}^R \in L^2_{\hat{\mu},0}(\hat{\Omega})\},
\end{multline}
in a slight abuse of nomenclature\footnote{Technically, $\mathcal{M}_R = V_h Y_0 \oplus \mathcal{M}_R^{\star}$, where $\mathcal{M}_R^{\star}$ is the manifold of rank-$R$, zero-mean random fields, see e.g.~\cite{muno18} for a more detailed discussion on the construction.}.
The representation of $R$-rank functions is not unique; conversely, a unique representation of tangent vectors $ \delta u = \delta u_0 Y_0 + \sum_{i=1}^R ( \delta u_i Y_i + U_i \delta y_i ) \in \mathcal{T}_u \mathcal{M}_R$ for $u =  \sum_{i=0}^R U_i Y_i$ is possible. In this work, we use the \textit{Dual Dynamically Orthogonal} (Dual DO) gauge condition~\cite{muno18}, which imposes 
\begin{equation} \label{eqn:dualDOconds}
	\mathbb{E}[Y_i \delta y_j] = 0, \quad \forall i,j = 1, \ldots,R.
\end{equation}
The \textit{tangent space at} $u = \sum_{i=0}^R U_i Y_i \in \mathcal{M}_{R}$ is then characterised as
\begin{multline}
	\mathcal{T}_u \mathcal{M}_{R} = \{\delta u = \delta u_0 Y_0 + \sum_{i=1}^R (\delta u_i Y_i + U_i \delta y_i), \text{ such that } \{\delta u_i\}_{i=0}^R \in V_h, \{\delta y_i\}_{i=1}^R \in L^2_{\hmu,0}(\hat{\Omega}) \\ 
	\mathbb{E}_{\hmu} [\delta y_i Y_j] = 0, \, \forall 1 \leq i,j \leq R
	\}.
\end{multline}
The tangent space is independent of $U_0$. Therefore, with a slight abuse of notation, given $U_0$, $U = (U_1, \ldots, U_R)$ and $Y = (Y_1, \ldots, Y_R)$ s.t. $u = U_0 Y_0 + U Y^{\top}$, we denote the tangent space at $u$ by $\mathcal{T}_{\!\!U Y^{\top}} \mathcal{M}_{R}$. Unless explicitly stated otherwise, $U$ always represents the $R$ deterministic modes $(U_1, \ldots, U_R)$ (without the mean $U_0$) and analogously for $Y = (Y_1, \ldots, Y_R)$. Introducing some further notation, $\mathcal{Y}$ denotes $\mathrm{span}\{Y_i\}_{i=1}^R \subset L^2_{\hat{\mu}, 0}(\hat{\Omega})$, and $\mathcal{P}_{\mathcal{Y}}[\cdot] = \sum_{i=1}^R \mathbb{E}_{\hmu}[\cdot Y_i ] Y_i$ the orthogonal projection on $\mathcal{Y}$. The projection on the orthogonal complement in $L^2_{\hat{\mu},0}$ is given by $\mathcal{P}^{\perp}_{\mathcal{Y}}[v] = v^{\star} - \mathcal{P}_{\mathcal{Y}}[v]$, where we recall that $v^{\star} = v - \mathbb{E}_{\hmu}[v]$.

The DLR Approximation framework consists in solving: 
\begin{dlrprob}
	\textit{Find } $u_{\mathrm{DLR}} : (0,T] \rightarrow \mathcal{M}_{R}$ \textit{s.t.} 
\begin{equation} \label{eqn:dlr-weak-form}
	(\dot{u}_{\mathrm{DLR}}, v_{\hhmu})_{X,\Lthm} + \mathcal{A}(u_{\mathrm{DLR}}, v_{\hhmu}) = \mathcal{F}(v_{\hhmu}) \quad \forall v_{\hhmu} \in \mathcal{T}_{u_{\mathrm{DLR}}} \mathcal{M}_{R}, \; \mathrm{a.e.}\; t \in (0,T]. 
\end{equation}
\end{dlrprob}
At each point in time, the test space depends on the solution itself. The DLR problem has a natural interpretation of projecting the dynamics $- L u_{\mathrm{DLR}} + f$ onto $\mathcal{T}_{u_{\mathrm{DLR}}}\mathcal{M}_R$. 

Finally, following the steps as in e.g.~\cite{sale09}, the solution $u_{\mathrm{DLR}} = U_0 Y_0 + \sum_{i=1}^R U_i Y_i$ of~\eqref{eqn:dlr-weak-form} is equivalently characterised by the solution to the system
\begin{DOeqs}
	\begin{align}	
&		\langle \dot{U}_i, v_h \rangle_X + \left\langle \mathbb{E}_{\hmu}[Lu_{\mathrm{DLR}} Y_i], v_h\right\rangle_{V^{\prime}V} = \langle \mathbb{E}_{\hat{\mu}}[f Y_i], v_h \rangle_{X}, && \forall v_h \in V_h, \forall 0 \leq i \leq R \label{eqn:DO-physical} \\
&		\sum_{i=1}^R\dot{Y}_i M_{ij} = \mathcal{P}^{\perp}_{\mathcal{Y}}[\left\langle f - L u_{\mathrm{DLR}} , U_j \right\rangle_{V^{\prime}V}],  && \text{ in } L^2_{\hmu,0}(\hat{\Omega}), \forall 1 \leq j \leq R \label{eqn:DO-stochastic}
	\end{align}
\end{DOeqs}
\noindent where $M_{ij} = \langle U_i, U_j \rangle_X$ denotes the Gram matrix of the basis $\{U_i\}_{i=1}^R$. 

\begin{remark}
	The low-rank approximation obtained by fixing $Y_0 = \mathbbm{1}$ results in a sub-optimal $R+1$ approximation, but some subspaces have inherent physical meaning and tracking the dynamics on those subspaces explicitly may be of particular interest. 
	For instance, in the context of the Vlasov-Poisson equation, this approach was used to ensure the conservation of mass, momentum and energy~\cite{eijo21}. In~\cite{muno18}, the authors use similar ideas to treat random boundary conditions.
\end{remark}

\section{A Petrov-Galerkin framework for DLR} \label{sec:pg-dlr}

In this section we combine the generalised Petrov-Galerkin formalism with the model order reduction provided by the DLR framework. 
Our approach is to start from the generalised Petrov-Galerkin formulation~\eqref{eqn:weak-form-pg} and restrict the space of test functions to $\mathcal{T}_{u_{\mathrm{DLR}}}\mathcal{M}_{R}$. This leads to the following generalised Petrov-Galerkin problem: 
\begin{sdlrprob} \noindent \textit{Find} $u_{\mathrm{PGDLR}} : (0,T] \rightarrow \mathcal{M}_{R}$ \textit{s.t.} 
\begin{multline} \label{eqn:PG-start}
	(\dot{u}_{\mathrm{PGDLR}}, \Hcal v_{\hhmu})_{X,\Lthm} + \mathcal{A}_h( u_{\mathrm{PGDLR}}, v_{\hhmu}) = \mathcal{F}(\Hcal v_{\hhmu}),
	\\ 
	\forall v_{\hhmu} \in \mathcal{T}_{u_{\mathrm{PGDLR}}}\mathcal{M}_{R}, \mathrm{a.e.}\; t \in (0,T].
\end{multline}
\end{sdlrprob}
It it noteworthy that the framework seeks a DLR approximation of a stabilised problem, rather than seeking to stabilise a DLR approximation. 
This framework has a natural interpretation of \textit{obliquely projecting} the dynamics $- L u_{\mathrm{DLR}} + f$ onto $\mathcal{T}_{u}\mathcal{M}_{R}$, unlike the standard DLR setting where the projections are orthogonal. 
We mention~\cite{dopanafeba23}, which uses a different type of oblique projections that arise from sampling columns and rows through the DEIM method, to efficiently compute the solution to DLR-like systems.   
Henceforth, the subscript ``PGDLR'' will be omitted to lighten the notation.

\subsection{On the skewing operator $\Hcal$} \label{subsubsec:Hcal}

The explicit characterisation of the time-derivatives of the physical and stochastic modes as in~\eqref{eqn:DO-physical}-\eqref{eqn:DO-stochastic} is a desirable feature of the DLR framework for multiple reasons. 
In the standard DO/DLR formulation~\eqref{eqn:DO-physical}-\eqref{eqn:DO-stochastic}, the time-derivatives have dimension $N_h$ for the physical modes (resp. $N_C$ for the stochastic modes), and naturally belong to $V_h$ (resp. $L^2_{\hmu,0}(\hat{\Omega})$); their computation does not require the assembly at each time step of an explicit basis of the tangent space; these parametrisations allow for particularly efficient computions when the operator $L$ and forcing term $f$ have a separable form (in the sense of e.g.~\cite[Definition 3.2]{cosc23}). This section explores when these properties can be recovered in the PG-DLR context.  

In the standard DLR framework, these properties hinge on two features of the tangent space in which the solution's time derivative is sought, namely (a) the tensor-product structure of the scalar product in $L^2_{\hat{\mu}}(\hat{\Omega}, X)$, (b) the unique parametrisation of tangent space elements by the orthonormality of the stochastic modes and the Dual DO condition. Specifically, the ``mass matrix'' of the tangent space (i.e., the matrix that realises the $L^2_{\hat{\mu}}(\hat{\Omega}, X)$ inner product on the basis introduced below)  is easy to invert. Taking a basis for the finite dimensional space $V_h = \mathrm{span}\{\phi^{i}\}_{i=1}^{N_h}$ and a basis for the stochastic space $ \mathrm{span}(\{Y_i\}_{i=0}^R \cup \{z^{j}\}_{j=1}^{N_C - R - 1})$ such that it forms an orthonormal basis of $L^2_{\hmu}(\hat{\Omega})$, a basis for the tangent space is given by $\{\phi^{i} Y_j  \} \cup \{U_l z^{k}\}$ with $j = 0,\ldots, R$, $l=1,\ldots,R$, $i=1, \ldots, N_h$, $k = 1, \ldots, N_C - R - 1 $. Arranging these test functions in the right order, the ``mass matrix'' of the tangent space is given by 
\begin{equation*}
	\mathbf{M}(\mathbf{u}) =
	\begin{bmatrix}
		I_{(R+1)} \otimes M_{V_h} & \mathbf{0} \\
		\mathbf{0} & M \otimes I_{N_C - R - 1},
	\end{bmatrix}
\end{equation*}
where $M_{V_h}$ is the mass matrix of the space $V_h$ obtained through the basis $\{\phi^i\}_{i=1}^{N_h}$, and $M$ is the Gram matrix $M_{ij} = \langle U_i, U_j \rangle_X$. Stacking the degrees of freedom of $\{\dot{U}_i\}_{i=0}^R \cup \{\dot{Y}_i\}_{i=1}^R$ in a vector $\dot{\mathbf{u}}$, solving problem~\eqref{eqn:dlr-weak-form} corresponds to solving the algebraic problem
\begin{equation*}
	\mathbf{M}(\mathbf{u}) \dot{\mathbf{u}} + \mathbf{A}(\mathbf{u}) = \mathbf{f}.
\end{equation*}
The inverse of $\mathbf{M}$ is easily computable as the matrix is block-diagonal and $(A \otimes B)^{-1} = A^{-1} \otimes B^{-1}$. Therefore,
the mass matrix can be explicitly inverted yielding explicit evolution equations for the individual modes. Note that the Kronecker-product structure is immediately inherited from the tensor-product scalar product, while the sparsity is inherited from the orthogonality of the stochastic basis. 

The tangent space skewed by $\Hcal$ in~\eqref{eqn:PG-start} permits to preserve (a) and (b) if $\Hcal$ verifies the conditions summarised below.

\begin{tancond}
	\leavevmode
	\begin{enumerate}[label=(C\arabic*)]
		\item \label{con:sep} $\Hcal$ is separable: $\Hcal = \Hcal_1 \otimes \Hcal_2$, where $\Hcal_1 : V_h \rightarrow X_h$ and $\Hcal_2 : L^2_{\hat{\mu}}(\hat{\Omega}) \rightarrow L^2_{\hat{\mu}}(\hat{\Omega})$. 
		\item \label{con:sym} $\Hcal_2$ is a symmetric, positive-definite operator on $\Lthm$ such that it induces a scalar product, in which case it would be natural to consider the alternative gauge condition 
\[
	\mathbb{E}_{\hmu}[\dot{Y}_i \Hcal_2 Y_j] = \mathbb{E}_{\Hcal_2}[\dot{Y_i} Y_j] = 0 \quad \forall 1 \leq i,j \leq R,
\]
instead of the standard one~\eqref{eqn:dualDOconds}.
\end{enumerate}
\end{tancond}

\begin{remark} \label{rem:noorthhyp}
	It is possible to forego either condition; then, one must effectively assemble the mass matrix and solve a system of size $R(N_h + N_C - R - 1) + N_h$ at each time step.
	This would imply constructing the orthonormal basis $\{z^{j}\}_{j=1}^{N_C - R}$ for the stochastic space at each time step. 
	The obtained update then being a tangent vector, one should also decide on a suitable retraction to the manifold, in the spirit of what was proposed in~\cite{kiva19}.
\end{remark}

The skewing conditions are quite restrictive and excludes even rather simple tangent space transformations, such as when 
\begin{equation} \label{eqn:sumHcal}
	\Hcal = \sum_{k=1}^K \Hcal_{1}^{(k)} \otimes \Hcal_{2}^{(k)}. 
\end{equation}
On the other hand, the form of these skewing operators is typically not overly complex when obtained from standard finite element stabilisations. 
In the case of SUPG and for $\bu$ deterministic, $\Hcal = (I+\alpha \bu \cdot \nabla) \otimes I = \Hcal_1 \otimes \Hcal_2$ naturally falls within this formulation. This stabilisation can also work for mildly stochastic transport, using the mean advection field as stabilising term, i.e. $\Hcal = (I + \alpha \mathbb{E}[\bu] \cdot \nabla) \otimes I$. Finally, for strongly stochastic advection fields, the proposed PG-DLR stabilisation approach is not expected to work well (which was verified through numerical examples not presented here) and it would then be of interest to consider the full stabilisation $\Hcal = I + \alpha \bu(x,\omega) \cdot \nabla$ by building and inverting the mass matrix at each iteration, as outlined in Remark~\ref{rem:noorthhyp}. This approach will not be addressed further in this work. 

The conditions~\ref{con:sep} and~\ref{con:sym} are specifically suited for the two-terms decomposition $UY^{\top}$ of the DO formalism, and not the three-term decomposition $\tilde{U}SV^{\top}$ found more commonly in the DLR literature~\cite{kolu07, luos14, celu22, eijo21, cosc23, cohu22}.
The absence of conditions on $\Hcal_1$ \modt{(as opposed to the condition on $\Hcal_2$ to be symmetric positive definite)} is a consequence of the fact that the deterministic modes $U$ are not subject to any constraint in the DO formalism, except that of linear independence. 

\begin{remark}
	It is also possible to swap the roles of the deterministic and stochastic modes, and impose the orthogonality conditions on the deterministic modes, implying that $\Hcal_1$ should be the operator that induces the new scalar product. However, many natural operators $\Hcal_1$ obtained from classical finite element stabilisations do not fall in this category.  
\end{remark}

\subsubsection{Uses of $\Hcal_2$}

In the basic framework developed here with a collocation method for the stochastic space, the possibility of skewing the stochastic space by some operator $\Hcal_2$ inducing a scalar product is not of much direct practical interest. It may however come useful in the field of Data Assimilation and particle filters, where the operator can act as a measure-reweighting operator. 

Modifying the stochastic space can also be beneficial for numerical stability when using a Galerkin method for the approximation of the random variables (e.g., polynomial approximations~\cite{doka02,ghsp91}), in the spirit of the approach in~\cite{dahusc12}. 
As an example, let $0 < c_{\min} \leq c(\omega) \leq c_{\max}$ for $\omega \in \Omega$, with $c_{\min} \ll 1$ and assume that for a significant subset $U$ of $\Omega$, $c(\omega) \ll 1$ . Furthermore, assume $L_{\mathrm{phys}}$ to be an operator that is coercive in some norm $\|\cdot\|_V$ stronger than $L^2(D)$, and with a moderate coercivity constant $\alpha$. Consider the problem
\begin{equation}
	\partial_t u + c(\omega) L_{\mathrm{phys}} u = f,
\end{equation}
with suitable initial and boundary values. 
Discretising $L^2_{\mu}(\Omega)$ by a Stochastic Galerkin method yields the problem : \textit{find } $u_N \in W$ \textit{s.t.}
\begin{equation} \label{eqn:toy-pb}
	(\partial_t u_N(\omega), w)_{X,L^2_{\mu}(\Omega)} + ( c(\omega) L_{\mathrm{phys}}  u_N(\omega), w)_{V'V,L^2_{\mu}(\Omega)}  = (f, w)_{X,L^2_{\mu}(\Omega)} ,\quad \forall w \in \tilde{W}.
\end{equation}	
The standard Galerkin method sets $W = \tilde{W} =  \mathbb{P}^N(\Omega, V)$ (a space of $V$-valued polynomial random variables) and testing against $u(\omega)$ yields a stability estimate of the form 
\begin{equation} \label{eqn:toy-prob-disc-1}
	\|u_N(T)\|^2_{L^2(D), L^2_{\mu}} + \alpha \int_{0}^T \|u_N(s)\|^2_{V,c} \mathrm{d}s 
	\leq 
	\|u_N(0)\|^2_{L^2(D), L^2_{\mu}}
	+
	\int_{0}^T \|f\|^2_{L^2_{\mu}} \mathrm{d}s,
\end{equation}
where $\| \cdot \|_{V,c} = \|c^{1/2} \cdot\|_{V,L^2_{\mu}}$. 
On the other hand, choosing $\tilde{W} = \Hcal P^N(\Omega, V) = \{ v/c(\omega) \;|\; v \in P^N(\Omega, V) \}$ and testing against $w = u(\omega) / c(\omega)$ yields the stability estimate
\begin{equation} \label{eqn:toy-prob-disc-2}
	\|u_N(T)\|^2_{L^2(D),c^{-1}} + \alpha \int_{0}^T \|u_N(s)\|_{V,L^2_{\mu}}^2 \mathrm{d}s 
	\leq 
	\|u_N(0)\|^2_{L^2(D),c^{-1}}
	+
	\int_{0}^T \|f\|^2_{L^2(D),c^{-1}} \mathrm{d}s.
\end{equation}

When performing a semi-discretisation by Galerkin method, error bounds for the error $e(t, \omega) = u(t,\omega) - u_h(t,\omega)$ can be obtained by writing an error equation for $e$ and re-using the stability estimate (replacing $u$ by $e$ in the left hand side of~\eqref{eqn:toy-prob-disc-1} and~\eqref{eqn:toy-prob-disc-2}). 

If for a significant subset $U$ of $\Omega$, $c(\omega)$ has a small value, the control over the
error term $\int_{0}^T \|e(s)\|^2_{V,c} \mathrm{d}s $
in~\eqref{eqn:toy-prob-disc-1} may not be very informative, as the error could
be relatively big in the region $U$. On the other hand, the corresponding error term in~\eqref{eqn:toy-prob-disc-2} uses the $V,L^2_{\mu}$-norm, hence
yielding a more uniform control in the $L^2((0,T), L^2(\Omega, V))$-norm. 


\subsection{PG-DLR equations}

In this section and the following ones, we restrict to the case $\Hcal_2 = I$, hence $\Hcal = \Hcal_1 \otimes I$.  
We are now in position of deriving practical equations to update the deterministic and stochastic modes of the PG-DLR problem~\eqref{eqn:PG-start}.
\begin{proposition}
	Evolution equations for PG-DLR (version 1). Find $\{\dot{U}_j\}_{j=0}^{R} \in (V_h)^{R+1}$ and $\{\dot{Y}_i\}_{i=1}^R \in (\mathcal{Y}^{\perp})^{R}$ such that
\begin{align}
	\lrX{\dot{U}_j, \Hcal_1 v_h} + \mathcal{A}_h(u, v_h Y_j )&= \mathcal{F}(\Hcal_1 v_h Y_j), &&\forall v_h \in V_h, \quad 0 \leq j \leq R, \label{eqn:PG-varf-detmodes-2} \\
	\sum_{i=1}^R\mathbb{E}_{\hmu}[\dot{Y}_i z] W_{ij} + \mathcal{A}_h(u, U_j \mathcal{P}^{\perp}_{\mathcal{Y}}z)  &=  \mathcal{F}(\Hcal_1 U_j \mathcal{P}_{\mathcal{Y}}^{\perp} z), &&\forall z \in \Lthm, \quad 1 \leq j \leq R \label{eqn:PG-varf-stochmodes-2},
\end{align}
where $W_{ij} = \lrX{U_i, \Hcal_1 U_j}$.
\end{proposition}
\begin{proof}
The equations are obtained by testing against respectively 
$v_h Y_i$ and $U_i z$ in~\eqref{eqn:PG-start}, where $v_h \in V_h$, $z \in \mathcal{Y}^{\perp}$. 
For~\eqref{eqn:PG-varf-detmodes-2}, testing against $v_h Y_j$, the first term becomes
\begin{multline}
	(\dot{U}_0 Y_0 + \sum_{i=1}^R \dot{U}_i Y_i + U_i \dot{Y}_i, \Hcal v_h Y_j)_{X, \Lthm} 
	=
	\langle  \mathbb{E}_{\hmu}[(\sum_{i=0}^R \dot{U}_i Y_i + U_i \dot{Y}_i)  Y_j] ,  \Hcal_1 v_h\rangle_{X} 
	\\
	= 
	\lrX{\dot{U}_j, \Hcal_1 v_h},
\end{multline}
having leveraged the orthogonality condition $\mathbb{E}_{\hmu}[Y_i Y_j] = \delta_{ij}$ and $\mathbb{E}_{\hmu}[Y_i \dot{Y}_j] = 0$. 
The other two terms already are as in~\eqref{eqn:PG-varf-detmodes-2}.
Equation \eqref{eqn:PG-varf-stochmodes-2} is obtained by testing against $U_j z$ for some $ z \in \mathcal{Y}^{\perp}$, or, equivalently, $U_j \mathcal{P}^{\perp}_{\mathcal{Y}} z$ with $z \in L^2_{\hmu}$ in~\eqref{eqn:PG-start}. The first term becomes
\begin{equation}
	(\dot{U}_0 Y_0 + \sum_{i=1}^R \dot{U}_i Y_i + U_i \dot{Y}_i, \Hcal U_j \mathcal{P}^{\perp}_{\mathcal{Y}}z)_{X, \Lthm} = 
	\sum_{i=1}^R \langle U_i, \Hcal_1 U_j \rangle_X \mathbb{E}_{\hmu}[{\dot{Y}_i z}]
	\\
	=
	\sum_{i=1}^R  \mathbb{E}_{\hmu}[{\dot{Y}_i z}] W_{ij},
\end{equation}
and again, the other two terms are already in the final form of~\eqref{eqn:PG-varf-stochmodes-2}.
\end{proof}

\subsection{Time-integration schemes} \label{subsubsec:pgdlrps}

A special care must be taken when performing time-discretisation of DLR systems, as standard time-integrators such as Runge-Kutta are known to be unstable~\cite{luki16}. 
Various time-integration schemes have been developed and studied in the DLR community, such as the Projector-Splitting scheme~\cite{luos14} or the BUG (or Unconventional) integrator~\cite{celu22}. 
In their basic form, those schemes work with a three-term decomposition $USV^{\top}$ and may not be straightforwardly adapted to the two-term decomposition $UY^{\top}$ used in the DO formalism. 
Here, we follow the approach proposed in~\cite{kavino21}, which relies on the approximation
\begin{equation}
	\frac{1}{\dt}\int_{t_n}^{t_{n+1}} \mathcal{A}_h(u(s), v) \mathrm{d}s \approx \mathcal{A}^1_h(u(t_n), v) + \mathcal{A}_h^2(u(t_{n+1}), v)
\end{equation}
This can represent any scheme among explicit Euler, implicit Euler, $\theta$-methods or IMEX-type schemes.  

In~\cite{kavino21}, the authors propose a staggered method, which computes first the deterministic modes, and then uses those to compute the update of the stochastic modes. 
It is not exactly equivalent to the famed Projector-Splitting scheme~\cite{luos14}, however it bears a very strong connection to it. 
The key difference resides in the fact that the former uses only one evaluation of the dynamics
(e.g. at $\uhr[n]$ for explicit Euler), while the latter sub-iterates in a splitting fashion.
\begin{alg} \label{alg:evaspl}
Given the solution $\uhr[n] = U_0^n Y_0 + \sum_{i=1}^R U_i^n Y_i^n$,
\begin{enumerate}
	\item Find $\tilde{U}^{n+1}_j \in \mathbb{R}^{N_h}$ such that
	\begin{multline}
		\dt^{-1}\lrX{ \tilde{U}^{n+1}_j - U^n_j, \Hcal_1 v_{h}} + \mathcal{A}_h^1(\uhr[n], v_h Y^n_j) + \mathcal{A}_h^2(\uhr[n+1], v_h Y^n_j) = \mathcal{F}(\Hcal_1 v_h Y^n_j), \\  \forall v_{h} \in V_h, \quad 0 \leq j \leq R. \label{eqn:PG-disc-varf-detmodes-3}
\end{multline}
	\item \label{itm:final-sol} Find $\widetilde{\delta{Y}}^{n+1}_j \in (\mathcal{Y}^{n})^{\perp}$ 
		such that
		\begin{multline} 
			\dt^{-1} \sum_{i=1}^R \mathbb{E}_{\hmu}[\widetilde{\delta Y}^{n+1}_j z] \tilde{W}^{n+1}_{ij} + \mathcal{A}_h^1(\uhr[n], \tilde{U}^{n+1} \mathcal{P}^{\perp}_{\mathcal{Y}^n}z) 
			+ \mathcal{A}_h^2(\uhr[n+1], \tilde{U}^{n+1}_j \mathcal{P}^{\perp}_{\mathcal{Y}^n} z) 
			\\
			= \mathcal{F}(\Hcal_1 \tilde{U}^{n+1}_j \mathcal{P}^{\perp}_{\mathcal{Y}^n}z),\quad
			\forall z \in \Lthm, \quad 1 \leq j \leq R. \label{eqn:PG-disc-vard-stochmodes-3}
		\end{multline} \label{itm:stoch-step}
	\item Set $\tilde{Y}^{n+1}_j = Y^n + \widetilde{\delta{Y}}^{n+1}_j$ for $ j = 1, \ldots, R$, \label{itm:set-Yn1}
	\item Find~$\{{U}^{n+1}_i\}_{i=1}^R$ and $\{Y^{n+1}_i\}_{i=1}^R$ such that $\sum_{i=1}^R  U^{n+1}_i Y^{n+1}_i = \sum_{i=1}^R \tilde{U}^{n+1}_i \tilde{Y}^{n+1}_i$ and $\mathbb{E}_{\hmu}[{Y^{n+1}_i Y^{n+1}_j}] = \delta_{ij}$.
		\label{itm:reorth-step}
	\item The new solution is given by $\uhr[n+1] = \tilde{U}^{n+1}_0 Y_0 + \sum_{i=1}^R U^{n+1}_i Y^{n+1}_i$.	
\end{enumerate}
\end{alg}
If $\tilde{W}^{n+1}$ were not invertible, one could instead solve~\eqref{eqn:PG-disc-vard-stochmodes-3} in a minimal-norm least squares fashion, as discussed in~\cite{kavino21}. 
In what follows however, we assume that $\tilde{W}^{n+1}$ is invertible for all $n$. 

In practice, it is not necessary to explicitly build the basis $(\mathcal{Y}^n)^{\perp} = \mathcal{P}^{\perp}_{\mathcal{Y}^n} (\Lthm)$ to solve~\eqref{eqn:PG-disc-vard-stochmodes-3}. 
Testing against $(\mathcal{Y}^n)^{\perp}$ is equivalent to (orthogonally) projecting the operator on $(\mathcal{Y}^n)^{\perp}$ by use of $\mathcal{P}^{\perp}_{\mathcal{Y}^n}$. Mathematically, 
\begin{equation}
	\mathcal{A}_h(u, U_i \mathcal{P}^{\perp}_{\mathcal{Y}^n} z) = (L_h u, U_i \mathcal{P}^{\perp}_{\mathcal{Y}^n} z) = \mathbb{E}_{\hmu}\left[ 
		\mathcal{P}^{\perp}_{\mathcal{Y}^n}[ \langle L_h u, U_i \rangle_{V_h'V_h} ] z 
	\right].
\end{equation}
Hence, as is done in Section~\ref{sec:supg-dlr} below, the problem can be directly solved in $\Lthm$. 

As was the case in~\cite{kavino21}, a discrete analog of the DO conditions holds for our discrete PG-DLR formulation,
and $u^n$, $u^{n+1}$ belong to the tangent space $\mathcal{T}_{\tilde{U}^{n+1}{Y^n}^{\top}} \mathcal{M}_{R,\hmu}$. This tool proves crucial when
establishing the discrete variational formulation. The proofs largely follow the ones in~\cite{kavino21}.

\begin{lemma} \label{lem:discdualdo}
	After one step of Algorithm~\ref{alg:evaspl}, the stochastic modes verify the following orthogonality properties: 
	\begin{enumerate}
		\item $\mathbb{E}_{\hmu}[{(\tilde{Y}^{n+1}_i - Y^n_i) Y^n_j}] = 0$ for $1 \leq i,j \leq R$, 
		\item $\mathbb{E}_{\hmu}[{\tilde{Y}^{n+1}_i Y^n_j}] = \delta_{ij}$. 
	\end{enumerate}
\end{lemma}
\begin{proof}
	1. and 2. follow immediately from the fact that $\widetilde{\delta Y}^{n}_j \in (\mathcal{Y}^n)^{\perp}$, and $\{Y_i\}_{i=1}^R$ is assumed orthonormal. 
\end{proof}

\begin{theorem} \label{th:disc-varf}
The following discrete variational formulation holds: 
\begin{multline} \label{eqn:discrete-varf-pg-2}
	\frac{1}{\dt}(\uhr[n+1] - \uhr[n], \Hcal v_{\hhmu})_{X,\Lthm} 
	+ \mathcal{A}_h^1(\uhr[n], v_{\hhmu}) + \mathcal{A}_h^2(\uhr[n+1], v_{\hhmu}) = \mathcal{F}(\Hcal v_{\hhmu}), \\ 
	\forall v_{\hhmu} \in \mathcal{T}_{\tilde{U}^{n+1}{(Y^{n})^\top}} \mathcal{M}_R,
\end{multline}

\end{theorem}
\begin{proof}
	Start by rewriting the first term of~\eqref{eqn:PG-disc-varf-detmodes-3} as $\dt^{-1} (\tilde{U}^{n+1}_j Y^n_j - U^n_j Y^n_j, \Hcal_1 v_h^{(j)} Y^n_j)_{X,\Lthm}$ for $1 \leq j \leq R$. 
	Owing to the orthogonality of the $Y_j$'s, summing~\eqref{eqn:PG-disc-varf-detmodes-3} over $j = 0, \ldots, R$ yields 
	\begin{multline} \label{eqn:varf-proof-1}
		\dt^{-1}\sum_{j=0}^R (\tilde{U}^{n+1}_j \hat{Y}^n_j - U^n_j \hat{Y}^n_j, \Hcal V (\hat{Y}^n)^{\top})_{X,\Lthm} 
		+ \mathcal{A}_h^1(\uhr[n], V (\hat{Y}^n)^{\top}) 
		+ \mathcal{A}_h^2(\uhr[n+1], V (\hat{Y}^n)^{\top})
		\\
		= \mathcal{F}(\Hcal V (\hat{Y}^n)^{\top}),
	\end{multline}
	where $V = (v_h^{(0)}, \ldots, v_h^{(R)})$ and $\hat{Y}^n = (Y^n_0, \ldots, Y^n_R)$.
	Similarly for~\eqref{eqn:PG-disc-vard-stochmodes-3}, we rewrite the first term as $\sum_{i=1}^R (\tilde{U}^{n+1}_i \tilde{Y}^{n+1}_i - \tilde{U}^{n+1}_i Y^n_i, \Hcal \tilde{U}^{n+1}_j z^{(j)})_{X,\Lthm}$. Again summing over $j$ yields
	\begin{multline} \label{eqn:varf-proof-2}
		\dt^{-1} \sum_{i=1}^R (\tilde{U}^{n+1}_i \tilde{Y}^{n+1}_i - \tilde{U}^{n+1}_i Y^n_i, \Hcal \tilde{U}^{n+1} Z^{\top})_{X,\Lthm}   
		+ \mathcal{A}_h^1(\uhr[n], \tilde{U}^{n+1} Z^{\top}) 
		+ \mathcal{A}_h^2(\uhr[n+1], \tilde{U}^{n+1} Z^{\top})
		\\
		= \mathcal{F}(\Hcal \tilde{U}^{n+1} Z^{\top}),
	\end{multline}
	where $Z = (\mathcal{P}^{\perp}_{\mathcal{Y}^n} z^{(1)}, \ldots, \mathcal{P}^{\perp}_{\mathcal{Y}^n} z^{(R)})$.  
	Now define $v_{\hhmu} = V (\hat{Y}^n)^{\top} + \tilde{U}^{n+1} Z^{\top} \in \mathcal{T}_{\tilde{U}^{n+1} (Y^n)^{\top}} \mathcal{M}_R$. 
	Summing~\eqref{eqn:varf-proof-1} and~\eqref{eqn:varf-proof-2} finishes the proof. 
	Indeed, denoting $\hat{u}^n = \tilde{U}^{n+1}_0 Y_0 + \sum_{j=1}^R \tilde{U}^{n+1}_j Y^{n}_j$, 
	the sum of the first two terms in~\eqref{eqn:varf-proof-1} and~\eqref{eqn:varf-proof-2} verifies
	\begin{multline}
\sum_{j=1}^R (\tilde{U}^{n+1}_j \tilde{Y}^{n+1}_j - \tilde{U}^{n+1}_j Y^n_j, \Hcal \tilde{U}^{n+1} Z^{\top})_{X,\Lthm} 
+
\sum_{j=0}^R (\tilde{U}^{n+1}_j \hat{Y}^n_j - U^n_j \hat{Y}^n_j, \Hcal V (\hat{Y}^n)^{\top})_{X,\Lthm} 
\\
= ( (\uhr[n+1] - \hat{u}^n)^{\star}, \Hcal \tilde{U}^{n+1} Z^{\top})_{X,\Lthm} + (\hat{u}^n - \uhr[n], \Hcal V (\hat{Y}^n)^{\top})_{X,\Lthm}
= (\uhr[n+1] - \uhr[n], v_{\hhmu}),
	\end{multline}
	as the cross-terms $( (\uhr[n+1] - \hat{u}^n)^{\star}, \Hcal V (\hat{Y}^n)^{\top})_{X,\Lthm}$ and $(\hat{u}^n - \uhr[n], \Hcal \tilde{U}^{n+1} Z^{\top})_{X,\Lthm}$ vanish:
	the former because $(\uhr[n+1] - \hat{u}^n)^{\star} = \sum_{i=1}^R \tilde{U}^{n+1}_i (\tilde{Y}^{n+1}_i - Y^n_i)$, and leveraging the orthogonality of $\tilde{Y}^{n+1} - Y^n$ and $Y^n$ by Lemma~\ref{lem:discdualdo}; 
	the latter because $\hat{u}^n - \uhr[n] = \sum_{i=0}^R \tilde{U}^{n+1}_i Y^n_i$, and $Z \subset (\mathcal{Y}^n)^{\perp}$.
Finally, $(\uhr[n+1] - \hat{u}^n)^{\star} + \hat{u}^n - \uhr[n] = \uhr[n+1] - \uhr[n]$. 
\end{proof}
Crucially, the following lemma holds.
\begin{lemma}
	Let $\uhr[n+1] = \tilde{U}_0^{n+1} Y_0 + \sum_{i=1}^R \tilde{U}^{n+1}_i \tilde{Y}^{n+1}_i$ be the output of Algorithm~\ref{alg:evaspl} after one step,
	starting from $\uhr[n] = U^n_0 Y^n_0 + \sum_{i=1}^R U^n_i Y^n_i$. Let $\hat{u} = \sum_{i=0}^R \tilde{U}_i^{n+1} Y^n_i$. 
	Then $\uhr[n]$, $\uhr[n+1] \in \mathcal{T}_{\hat{u}}\mathcal{M}_{R}$. 
\end{lemma}
\begin{proof}
	The proof is a direct check. 
\end{proof}

Just as in~\cite{kavino21}, one remarkable feature of the variational formulation~\eqref{eqn:discrete-varf-pg-2} is that this result holds independently of the potential ill-conditioning of the matrix $\tilde{W}^{n+1}$ that appears in step $2$ of Algorithm~\ref{alg:evaspl}. 
As $\tilde{W}^{n+1}$ is the ``skewed'' Gram matrix obtained from the basis $\{\tilde{U}^{n+1}_i\}_{i=1}^R$, this situation naturally arises when the chosen rank is such that some modes are close to be linearly dependent, i.e., redundant (rank over-approximation).   
However, through Theorem~\ref{th:disc-varf} and upon choosing a suitable time-discretisation, the numerical PG-DLR solutions can readily be proven to be (norm-)stable even in the presence of small singular values in $\tilde{W}^{n+1}$; this will be illustrated below for the case of SUPG-DLR. Algorithm~\ref{alg:evaspl} is thus understood to be robust to the smallest singular values.

\section{SUPG-stabilised DLR} \label{sec:supg-dlr}

This section details the application of the PG-DLR framework to the case of SUPG-stabilisation of unsteady advection-diffusion-reaction equations.


Throughout this section, we consider the case~$\bu$ deterministic, which alows to explicitly characterise the update equations of the physical and stochastic modes. 
Two time-discretisations are of interest, the implicit Euler and a semi-implicit one. 
The implicit Euler scheme, while being fully non-linear and computationally very heavy, brings about the idea of the discretised equations, and furthermore enjoys desirable norm-stability properties. 
The semi-implicit scheme, while more cumbersome to write, yields a practical scheme to compute the iterations of the DLR system, by reducing it to a sequence of linear systems to solve.

\subsubsection{Implicit Euler}

When using Algorithm~\ref{alg:evaspl}, the implicit Euler scheme corresponds to the choice 
$\mathcal{A}^1_h(u,v) = 0$, and $\mathcal{A}^2_h(u,v) = \aSUPG(u,v)$.
The equations for the DLR modes then read

\textit{Given $\uhr[n] = U_0^n Y_0^n + \sum_{i=1}^R U^n_i Y^n_i$, the solution at time $t_{n+1}$ is given by}

\textit{1. Deterministic modes: for $0 \leq j \leq R$, find $\tilde{U}^{n+1}_j \in V_h$ such that}
\begin{multline} \label{eqn:detmodes-adr-imp}
	\lrX{\frac{\tilde{U}^{n+1}_j - U^n_j}{\triangle t},  v_h} +  
	\lrX{\mathbb{E}[\varepsilon \nabla \uhr[n+1] Y^n_j], \nabla v_h} 
	+\lrX{ \mathbb{E}[\bu \cdot \nabla \uhr[n+1] Y^n_j], v_h} 
	+ \lrX{\mathbb{E}[c \uhr[n+1] Y^n_j], v_{\hhmu}}
	\\ 
	+ \sum_{K \in \Th} \frac{\delta_K}{\dt}\langle \tilde{U}^{n+1}_j - U^n_j, \bu \cdot \nabla v_h\rangle_K 
	+ \delta_K \langle \mathbb{E}[(- \varepsilon \Delta \uhr[n+1] + \bu \cdot \nabla \uhr[n+1] + c\uhr[n+1]) Y^n_j ] , \bu \cdot \nabla v_h \rangle_K 
	\\
	= \lrX{\mathbb{E}[f Y^n_j], v_h} + \sum_{K \in \Th} \delta_K \lrX{\mathbb{E}[f Y^n_j], \bu \cdot \nabla v_h} \quad \forall v_h \in V_h,   
\end{multline}

\textit{2. Stochastic modes: for $1 \leq j \leq R$, find $\tilde{Y}^{n+1}_j \in \mathbb{R}^{N_C}$ such that}
\begin{multline} \label{eqn:stochmodes-adr-imp}
 \sum_{i=1}^R	\frac{\tilde{Y}^{n+1}_i - Y^n_i}{\dt} \tilde{W}^{n+1}_{ij} 	
	+ \mathcal{P}_{\mathcal{Y}}^{\perp} [ \lrX{\varepsilon \nabla \uhr[n+1], \nabla \tilde{U}^{n+1}_j} 
	+ \lrX{\bu \cdot \nabla \uhr[n+1], \tilde{U}^{n+1}_j} 
	+ \lrX{c \uhr[n+1], \tilde{U}^{n+1}_j} 
	\\
	\sum_{K \in \Th}  \langle - \varepsilon \Delta \uhr[n+1] + \bu \cdot \nabla \uhr[n+1] + c \uhr[n+1], \bu \cdot \nabla \tilde{U}^{n+1}_j \rangle_K ] 
	\\
	= \mathcal{P}_{\mathcal{Y}}^{\perp}[ \lrX{f, \tilde{U}^{n+1}_j} 
	+
	\sum_{K \in \Th} \delta_K \langle f, \bu \cdot \nabla \tilde{U}^{n+1}_j \rangle_K 
	] \quad \mathrm{in} \quad \Lthm, 
\end{multline}
where $\tilde{W}^{n+1}_{ij} = \lrX{\tilde{U}^{n+1}_i, \tilde{U}^{n+1}_j} + \sum_{K \in \Th} \delta_K \langle \tilde{U}^{n+1}_i, \bu \cdot \nabla \tilde{U}^{n+1}_j \rangle_K$ for $1 \leq i,j \leq R$.

\textit{3. Orthonormalise basis : find $\{U_i^{n+1}\}_{i=0}^R$, $\{Y_i^{n+1}\}_{i=1}^R$ such that $U^{n+1}_0 = \tilde{U}^{n+1}_0$, 
	$\sum_{j=1}^R U^{n+1}_j Y^{n+1}_j = \sum_{j=1}^R \tilde{U}^{n+1}_j \tilde{Y}^{n+1}_j$, 
	and $\mathbb{E}_{\hmu}[Y^{n+1}_i Y^{n+1}_j] = \delta_{ij}$. 
}

The DLR iterates $\{\uhr[n]\}_{n=0}^N$ verify the variational formulation:  
\begin{multline} \label{eqn:varf-adr-disc-ie}
	\frac{1}{\dt}(\uhr[n+1] - \uhr[n], v_{\hhmu}) + \aSUPG(\uhr[n+1], v_{\hhmu}) + 	
	\sum_{K \in \Th} \frac{\delta_K}{\dt} (\uhr[n+1] - \uhr[n], \bu \cdot \nabla v_{\hhmu})_K  
	\\ 
	= (f, v_{\hhmu}) + \sum_{K \in \Th} \delta_K (f, \bu \cdot \nabla  v_{\hhmu})_K,  \quad \forall v_{\hhmu} =  v_{\hhmu} \in \mathcal{T}_{\!\!\tilde{U}^{n+1} (Y^n)^{\top}}\mathcal{M}_R.
\end{multline}
\subsubsection{Semi-Implicit}

The semi-implicit scheme, originally proposed in~\cite{kavino21}, is particularly interesting in the cases of mild stochasticity, as it leads to a cost similar to the explicit Euler scheme while featuring increased stability. 
\modt{It leverages the fact that a part of the stochastic space is fixed, which allows to treat a part of the right hand side implicitly when integrating without it inducing a coupled non-linear system; this leads to better norm-stability properties of the scheme.}
In our context, the scheme takes the form
\begin{align*}
	\mathcal{A}_h^1(u, v) &= (\varepsilon^{\star} \nabla u, \nabla v) 
	+ (c^{\star} u, v) + \sum_{K \in \Th} \delta_K (- \varepsilon^{\star} \Delta u 
	+  c^{\star}u, \bu \cdot \nabla v)_K 
	\\
	\mathcal{A}_h^2(u, v) &= (\bar{\varepsilon} \nabla u, \nabla v) + (\bu \cdot \nabla u, v) + (\bar{c} u, v) + \sum_{K \in \Th} \delta_K (- \bar{\varepsilon} \Delta u + \bu \nabla u + \bar{c}u, \bu \cdot \nabla v)_K, 
\end{align*}
where $\bar{\varepsilon} = \mathbb{E}_{\hmu}[\varepsilon]$, $\varepsilon^{\star} = \varepsilon - \bar{\varepsilon}$ (and similarly for $\bar{c}$ and $c^{\star}$).

As this version is used in our simulations, we present the practical evolution equations, after introducing two short lemmas that clarify how these equations are obtained: 
\begin{lemma}
	Assume that $\mathcal{W} : V_h \rightarrow V_h^{\prime}$ is a deterministic operator. Then, 
	\begin{equation}		\mathbb{E}[{ \langle \mathcal{W} \uhr[n+1], v \rangle_{V_h^{\prime}V_h} \mathcal{P}_{\mathcal{Y}}^{\perp} z}] = \mathbb{E}[{ \langle \mathcal{W} \tilde{U}^{n+1} (\tilde{Y}^{n+1} - Y^n)^{\top},v \rangle_{V_h^{\prime} V_h} z }] 
	\end{equation}
\end{lemma}
\begin{proof}
This is a simple consequence of the fact that, as 
	\[
	\uhr[n+1] = U^{n+1}_0 + \tilde{U}^{n+1} (\tilde{Y}^{n+1})^{\top} =  U^{n+1}_0 + \tilde{U}^{n+1} (\tilde{Y}^{n+1} - Y^n)^{\top} + \tilde{U}^{n+1} (Y^n)^{\top} ,
	\]
	consequently
	\[ 
		\mathcal{P}_{\mathcal{Y}}^{\perp} \mathcal{W} \uhr[n+1] = \mathcal{W} \mathcal{P}_{\mathcal{Y}}^{\perp} \uhr[n+1] = \mathcal{W} \tilde{U}^{n+1} (\tilde{Y}^{n+1} - Y^n)^{\top}.
	\]
\end{proof}

\begin{lemma}
	Assume that $\mathcal{W} : V_h \rightarrow V_h^{\prime}$ is a deterministic operator. Then,
	\begin{equation}
		 \langle \mathbb{E}[{\mathcal{W} \uhr[n+1], Y^n_j}] v_h  \rangle_{V^{\prime}_h V_h} = 
		\lrX{ \mathcal{W} \tilde{U}^{n+1}_j, v_h }
	\end{equation}
\end{lemma}
\begin{proof}
\begin{equation*}
	\langle \mathbb{E}[{\mathcal{W} \uhr[n+1] Y^n_j}], v_h \rangle_{V^{\prime}_h V_h} 
	=
	\sum_{i=0}^R \lrX{ \mathcal{W} \tilde{U}^{n+1}_i \mathbb{E}[{\tilde{Y}^{n+1}_i Y^n_j}] v_h} 
	\\
	=
	\lrX{ \mathcal{W} \tilde{U}^{n+1}_j, v_h },
\end{equation*}
having used Lemma~\ref{lem:discdualdo}.
\end{proof}
The update equations then read as follows: 

\textit{Given $\uhr[n] = U_0^n Y_0^n + \sum_{i=1}^R U^n_i Y^n_i$, compute} 

\textit{1. Deterministic modes: for $0 \leq j \leq R$, find $\tilde{U}^{n+1}_j \in V_h$ s.t.}
\begin{multline} \label{eqn:detmodes-adr-si}
	\lrX{\frac{\tilde{U}^{n+1}_j - U^n_j}{\triangle t},v_h} 	
	+ \lrX{ \bar{\varepsilon} \nabla \tilde{U}^{n+1}_j, \nabla v_h} 	
	+\lrX{\bu \cdot \nabla \tilde{U}^{n+1}_j, v_h} 
	+ \lrX{\bar{c} \tilde{U}^{n+1}_j, v_h} 
	\\
	+ \sum_{K \in \Th} \frac{\delta_K}{\dt}\langle \tilde{U}^{n+1}_j - U^n_j, \bu \cdot \nabla v_h\rangle_K 
	+ \delta_K \langle 
	-\bar{\varepsilon} \Delta \tilde{U}^{n+1}_j 
	+ \bu \cdot \nabla \tilde{U}^{n+1}_j  
	+ \bar{c} \tilde{U}^{n+1}_j ,
	\bu \cdot \nabla v_h \rangle_K  
	\\ 
	+ \lrX{ \mathbb{E}[\varepsilon^{\star} \nabla \uhr[n] Y^n_j], \nabla v_h} 
	+ \lrX{ \mathbb{E}[c^{\star} \uhr[n]Y^n_j ], v_h}  
	+ \sum_{K \in \Th} \delta_K \langle 
	\mathbb{E}[(- \varepsilon^{\star} \Delta \uhr[n] 
	+ c^{\star} \uhr[n] )Y^n_j ],
	\bu \cdot \nabla v_h \rangle_K  
	\\
	=
	\lrX{\mathbb{E}[f Y^n_j], v_h} 
	+ \sum_{K \in \Th} \langle \mathbb{E}[f Y^n_j], \bu \cdot \nabla v_h \rangle_K,
	\quad \forall v_h \in V_h   
\end{multline}

\textit{2. Stochastic modes: for $1 \leq j \leq R$, find $\tilde{Y}^{n+1}_j \in \mathbb{R}^{N_C}$ s.t.}
\begin{multline} \label{eqn:stochmodes-adr-si}
	\sum_{i=1}^R	(\tilde{Y}_i^{n+1} - Y^n_i) 
	\hat{W}_{ij}
+ \mathcal{P}_{\mathcal{Y}}^{\perp} 
	[ 
	\lrX{\varepsilon^{\star} \nabla \uhr[n], \nabla \tilde{U}^{n+1}_j}
	+ \lrX{c^{\star} \uhr[n+1], \tilde{U}^{n+1}_j} ]
	\\
	+ \mathcal{P}_{\mathcal{Y}}^{\perp} \left[
	\sum_{K \in \Th} \delta_K \langle - \varepsilon^{\star} \Delta \uhr[n] + c^{\star} \uhr[n], \bu \cdot \nabla \tilde{U}^{n+1}_j\rangle_K \right]
	= \mathcal{P}_{\mathcal{Y}}^{\perp} \left[
	\lrX{f, \tilde{U}^{n+1}_j}
	+ \sum_{K \in \Th} \delta_K \langle f, \tilde{U}^{n+1}_j \rangle_{K}
	\right]
\end{multline}
where 
\begin{multline*}
	\hat{W}_{ij} = \frac{\tilde{W}_{ij}}{\dt} 
	+ \lrX{\bar{\varepsilon} \nabla \tilde{U}^{n+1}_i, \nabla \tilde{U}^{n+1}_j} 
	+ \lrX{ \bu \cdot \nabla \tilde{U}^{n+1}_i, \tilde{U}^{n+1}_j} 
	\\
	+ \sum_{K \in \Th} \delta_K 
	\langle 
	- \bar{\varepsilon} \Delta \tilde{U}^{n+1}_i
	+ \bu \cdot \nabla \tilde{U}^{n+1}_i  
	+ \bar{c} \tilde{U}^{n+1}_i
	, \bu \cdot \nabla \tilde{U}^{n+1}_j \rangle_K 
	),
\end{multline*}
and again $\tilde{W}^{n+1}_{ij} = \lrX{\tilde{U}^{n+1}_i, \tilde{U}^{n+1}_j} + \sum_{K \in \Th} \delta_K \langle \tilde{U}^{n+1}_i, \bu \cdot \nabla \tilde{U}^{n+1}_j \rangle_K$ for $1 \leq i,j \leq R$.

\textit{3. Orthonormalise basis as in step $3$ of Implicit Euler.}
\vspace{1em}


The semi-implicit scheme gives rise to the following variational formulation 
\begin{multline} \label{eqn:varf-adr-disc-si}
	\frac{1}{\dt}(\uhr[n+1] - \uhr[n], v_{\hhmu}) + (\bar{\varepsilon} \nabla \uhr[n+1], \nabla v_{\hhmu}) 
	+ (\varepsilon^{\star} \nabla \uhr[n], \nabla v_{\hhmu})  
	+  (\bu \cdot \nabla \uhr[n+1], v_{\hhmu}) 
	 \\ 
	 + (\bar{c} \uhr[n+1], v_{\hhmu}) + (c^{\star} \uhr[n], v_{\hhmu}) +  \sum_{K \in \Th} \frac{\delta_K}{\dt} (\uhr[n+1] - \uhr[n], v_{\hhmu})_K 
	 \\ + \sum_{K \in \Th}  \delta_K (- \bar{\varepsilon} \Delta \uhr[n+1] - \varepsilon^{\star} \Delta u^n + \bu \cdot \nabla \uhr[n+1] + \bar{c} \uhr[n+1] + c^{\star}u^n, \bu \cdot \nabla v_{\hhmu})_K 
	\\
	= (f, v_{\hhmu}) + \sum_{K \in \Th} \delta_K (f, \bu \cdot \nabla  v_{\hhmu})_K, 
	\quad
	\forall v_{\hhmu} \in \mathcal{T}_{\!\!\tilde{U}^{n+1} (Y^n)^{\top}}\mathcal{M}_R 
\end{multline}

Step~$2$ of the algorithm involves inverting $\hat{W}$, which is possible for $\dt$ small enough under the assumption that $\tilde{W}$ is invertible. 

\subsection{Norm-stability of SUPG-stabilised DLR} \label{subsec:normstab}

This section establishes different norm-stability properties of the SUPG-stabilised DLR solution. 
%
 \label{subsec:ie-int}

\subsubsection{Implicit Euler discretisation}

\begin{theorem} \label{th:im-stab}
	Assume~\eqref{eqn:varepscoer},~\eqref{eqn:deltaK-coerc} and
	\begin{equation} \label{eqn:dK-dt-bound}
		\delta_K \leq \frac{\dt}{4}.
	\end{equation}
	If (i) $\mu_0 > 0$, it holds
	\begin{equation} \label{eqn:pure-bound}
		\| \uhr[N] \|^2_{} + \dt C_1 \sum_{n=0}^{N-1} \|\uhr[n+1]\|^2_{\mathrm{SUPG}} \leq \|\uhr[0]\|^2_{} + \dt
		C_2
		\sum_{n=0}^{N-1} \|f\|^2_{},
	\end{equation}
	with $C_1 = \nicefrac{1}{2}$ and $C_2 = \left(\frac{2}{\mu_0} + 4\delta\right)$. \\
	If \textit{(ii)} $\mu_0 = 0$ and $f = 0$, then~\eqref{eqn:pure-bound} holds true with with $C_1 = \nicefrac{3}{4}$ and $C_2 = 0$. \\
	If (iii) $\mu_0 = 0$ and $f \neq 0$, then for $\dt < (1 + 2\nu)^{-1}$, 
	it holds 	
	\begin{equation} \label{eqn:gw-bound}
		\| \uhr[N] \|^2_{} + \frac{\dt}{2} \sum_{n=0}^{N-1} \|\uhr[n+1]\|^2_{\mathrm{SUPG}} \leq C_3 \left( \|\uhr[0]\|^2_{} + \dt	
		\sum_{n=0}^{N-1} \|f\|^2_{} \right) ,
	\end{equation}
	with $C_3 = e^{(1 + 2 \nu) T}$.
\end{theorem}
\begin{proof}
	We first detail a few bounds, and then expand on how those can be used to obtain the different cases. The main ``challenge'' is to suitably bound the term $\delta_K \dt^{-1}(\uhr[n+1] - \uhr[n], \bu \cdot \nabla \uhr[n+1])_{K,\Lthm}$, for which we use the same approach as in~\cite{jovo11}. This is where condition~\eqref{eqn:dK-dt-bound} is needed.

	Testing~\eqref{eqn:varf-adr-disc-ie} against $\uhr[n+1] \in \mathcal{T}_{\tilde{U}^{n+1} {Y^n}^{\top}} \mathcal{M}_{\!R}$ and using the fact that $2(u-v,u) = \|u\|^2 - \|v\|^2 + \|u - v\|^2$, the first term becomes 
	
	\begin{equation} \label{eqn:un-diff-term}
		\frac{1}{2\dt}(\|\uhr[n+1] \|_{}^2 -  \| \uhr[n] \|_{}^2 +  \| \uhr[n+1]- \uhr[n] \|_{}^2). 
	\end{equation}
\noindent The second term verifies $\aSUPG(\uhr[n+1], \uhr[n+1]) \geq \nicefrac{1}{2} \|\uhr[n+1]\|^2_{\mathrm{SUPG}} - \nu \|\uhr[n+1]\|^2_{}$ by~\eqref{eqn:asupg-weak-coer}. 

\noindent For the second-to-last term, using Cauchy-Schwarz and the $\epsilon$-Young inequality gives
	\begin{align} \label{eqn:diff-un-bgrad}
	 & \left| 
	 \sum_{K \in \mathcal{T}_{\!\!h}} \frac{\delta_K}{\dt} (\uhr[n+1] - \uhr[n], \bu \cdot \nabla \uhr[n+1])_{K,\Lthm}
	\right|
	\\ 
	 & \qquad \leq 
	\sum_{K \in \mathcal{T}_{\!\!h}} \frac{\delta_K}{\dt} \left(\frac{2\| \uhr[n+1] - \uhr[n] \|^2_{K,\Lthm}}{\dt} + \frac{ \dt \| \bu \cdot \nabla \uhr[n+1]\|_{K,\Lthm}^2}{8} \right) \\
	& \qquad\leq \sum_{K \in \mathcal{T}_{\!\!h}} \frac{1}{2\dt} \| \uhr[n+1] - \uhr[n] \|^2_{K,\Lthm} +  \frac{\delta_K}{8} \| \bu \cdot \nabla \uhr[n+1]\|_{K,\Lthm}^2
	\\
	 & \qquad \leq \frac{1}{2\dt} \| \uhr[n+1] - \uhr[n] \|^2_{} + \frac{1}{8}\sum_{K \in \mathcal{T}_{\!\!h}}  \delta_K \| \bu \cdot \nabla \uhr[n+1]\|_{K,\Lthm}^2, 
	\end{align}
	where the third inequality is obtained using~\eqref{eqn:dK-dt-bound}.
	Finally, the terms in the right hand-side can be bounded in two ways. Assuming $\mu_0 > 0$,
	\begin{equation} \label{eqn:f-bound-muzero}
		|(f, \uhr[n+1])| \leq \frac{1}{\mu_0} \|f\|^2_{} + \frac{1}{4} \| \mu^{\nicefrac{1}{2}} \uhr[n+1] \|^2_{}.
	\end{equation}
	If $\mu_0 = 0$ however, we bound
	\begin{equation} \label{eqn:f-bound-nomu}
		|(f, \uhr[n+1])| \leq  \frac{1}{2}\|f\|^2_{} + \frac{1}{2} \| \uhr[n+1] \|^2_{}.
	\end{equation}
	Furthermore, 
	\begin{multline} \label{eqn:fbgrad-bound}
		|\sum_{K \in \Th} \delta_K (f, \bu \cdot \nabla \uhr[n+1])| 
		\leq \sum_{K \in \Th} 2 \delta_K \|f\|^{2}_{K,\Lthm} + \frac{\delta_K}{8} \|\bu \cdot \nabla \uhr[n+1]\|^2_{K,\Lthm} 
		\\
		\leq 2 \delta \|f\|^{2}_{} +  \frac{1}{8}\sum_{K \in \Th}\delta_K \|\bu \cdot \nabla \uhr[n+1]\|^2_{K,\Lthm},
	\end{multline}
	where $\delta \coloneqq \max_{K \in \Th} \delta_K$.
	Firstly, for the case $\mu_0 > 0$, we lower-bound the l.h.s of~\eqref{eqn:varf-adr-disc-ie} using~\eqref{eqn:un-diff-term} and the coercivity of $\aSUPG$ (note that $\nu = 0)$, and upper-bound the r.h.s with~\eqref{eqn:diff-un-bgrad},~\eqref{eqn:f-bound-muzero} and~\eqref{eqn:fbgrad-bound} to obtain
	\begin{multline}
		\frac{1}{2\dt}(\|\uhr[n+1] \|_{}^2 -  \| \uhr[n] \|_{}^2 +  \| \uhr[n+1]- \uhr[n] \|_{}^2) + \frac{1}{2} \|\uhr[n+1]\|^2_{\mathrm{SUPG}} 
		\\ 
		\leq \frac{1}{2\dt} \| \uhr[n+1] - \uhr[n] \|^2_{} + \frac{1}{4}\sum_{K \in \mathcal{T}_{\!\!h}}  \delta_K \| \bu \cdot \nabla \uhr[n+1]\|_{K,\Lthm}^2 + 
		\left(\frac{1}{\mu_0}+  2 \delta \right)\|f\|^2_{} 
		\\
		+ \frac{1}{4} \| \mu^{\nicefrac{1}{2}} \uhr[n+1] \|^2_{}.
	\end{multline}
	Claim \textit{(i)} follows by using 
	\begin{equation} \label{eqn:coer-bound1}
		\frac{1}{2} \|\uhr[n+1]\|^2_{\mathrm{SUPG}} - \frac{1}{4} \sum_{K \in \Th} \delta_K \| \bu \cdot \nabla \uhr[n+1]\|^2_{K,\Lthm} - \frac{1}{4} \|\mu^{\nicefrac{1}{2}} \uhr[n+1]\|^2_{} \geq \frac{1}{4}\|\uhr[n+1]\|^2_{\mathrm{SUPG}},
	\end{equation}
	rearranging a few terms, multiplying by $2\dt$ and summing over $n = 0, \ldots, N-1$. 

	Note that the assumption $\mu_0 >0$ is only used in~\eqref{eqn:f-bound-muzero} -- hence, if $\mu_0 = 0$ and $f = 0$, the other terms can be bounded in exactly the same way. We thus only need to control 
	\begin{equation} \label{eqn:coer-bound2}
		\frac{1}{2} \|\uhr[n+1]\|^2_{\mathrm{SUPG}} - \frac{1}{8} \sum_{K \in \Th} \delta_K \| \bu \cdot \nabla \uhr[n+1]\|^2_{K,\Lthm} \geq \frac{3}{8} \|\uhr[n+1]\|^2_{\mathrm{SUPG}},
	\end{equation}
	which subsequently yields $C_1 = \nicefrac{3}{4}$. This is claim \textit{(ii)}. 

	The more general approach to treat the case $\mu_0 = 0$ and $f \neq 0$ is to use~\eqref{eqn:f-bound-nomu} and the weak coercivity. Performing the same steps as before yields
	\begin{multline}
		\frac{1}{2\dt}(\|\uhr[n+1] \|_{}^2 -  \| \uhr[n] \|_{}^2 +  \| \uhr[n+1]- \uhr[n] \|_{}^2) + \frac{1}{2} \|\uhr[n+1]\|^2_{\mathrm{SUPG}} 
		\\ 
		\leq \frac{1}{2\dt} \| \uhr[n+1] - \uhr[n] \|^2_{} + \frac{1}{4}\sum_{K \in \mathcal{T}_{\!\!h}}  \delta_K \| \bu \cdot \nabla \uhr[n+1]\|_{K,\Lthm}^2 + \left(\frac{1}{2}  + 2 \delta\right) \|f\|^2_{} 
		\\+ \left( \frac{1}{2} + \nu  \right) \| \uhr[n+1] \|^2_{},
 \end{multline}
 which yields
	\begin{equation*}
		(1 - \dt(1+2\nu))\|\uhr[n+1] \|_{}^2 -  \| \uhr[n] \|_{}^2 + \frac{\dt}{4} \|\uhr[n+1]\|^2_{\mathrm{SUPG}}
		\leq \dt \left(1  + 4 \delta\right) \|f\|^2_{}.
 \end{equation*}
 The condition $\dt < (1+2\nu)^{-1}$ ensures the leftmost term is nonnegative, and using a standard Gronwall-lemma (e.g.~\cite{quva08}) yields~\eqref{eqn:gw-bound}. The constants are determined by the growth factor $(1 + 2 \nu)$. 
\end{proof}

\subsubsection{Semi-Implicit integrator} \label{subsec:si-int}

The desirable properties of the semi-implicit integrator become apparent under moderate stochasticity; this is expressed through the following assumptions:
\begin{align}
&	|\varepsilon^{\star}(\omega)| \leq \frac{1}{32} \hat{\varepsilon}, && \forall \omega \in \hat{\Omega}  \label{eqn:varepscond}
	\\
&	|c^{\star}(x,\omega)|  \leq \frac{1}{32} \mu(x,\omega), && \forall x \in D, \omega \in \hat{\Omega} \label{eqn:ccond} 
\end{align}
The latter condition relates $c^{\star}$ to $\mu$ rather than some quantity only dependent on $c$ (even if, in the case of zero-divergence and $c$ positive, $\mu = \frac{1}{2}c$). 
Furthermore, while not presented here, it is also possible to ask a more general condition $|c^{\star}(x,\omega)| \leq C_1 c(x,\omega)$, and find Gronwall-type bounds. 

We remark that those conditions are merely sufficient to ensure the stability results we propose; as can be seen in the proof of Theorem~\ref{th:si-stab}, there is some flexibility and terms may be balanced in different ways.

\begin{theorem} \label{th:si-stab}
	Assume~\eqref{eqn:varepscoer} and 
	\begin{equation} \label{eqn:si-dK-bound}
		\delta_K \leq \frac{1}{8}\left\{ \frac{1}{2 \trnorm{c}_K}, \frac{h_K^2}{2 \hat{\varepsilon} C_I^2 \max(C^2_E, 1) d}, 2\dt \right\}.
	\end{equation}
Let $A_0 = \|\uhr[0]\|^2_{} + \frac{1}{8} \hat{\varepsilon} \|\nabla \uhr[0]\|^2_{} +  \frac{1}{8} \|\mu^{\nicefrac{1}{2} }\uhr[0]\|^2_{}$.

	If (i) $\mu_0 > 0$, it holds
	\begin{equation} \label{eqn:si-pure-bound}
		\| \uhr[N] \|^2_{} + \dt C_1 \sum_{n=0}^{N-1} \|\uhr[n+1]\|^2_{\mathrm{SUPG}} \leq A_0 + \dt
		C_2
		\sum_{n=0}^{N-1} \|f\|^2_{},
	\end{equation}
	with $C_1 = \nicefrac{1}{4}$ and $C_2 = \left(\frac{2}{\mu_0} + 4\delta\right)$. \\ 
	If \textit{(ii)} $\mu_0 = 0$ and $f = 0$, then~\eqref{eqn:pure-bound} holds true with with $C_1 = \nicefrac{1}{2}$ and $C_2 = 0$.  \\
	If (iii) $\mu_0 = 0$ and $f \neq 0$, then for $\dt < (1 + 2\nu)^{-1}$,
	\begin{equation} \label{eqn:si-gw-bound}
		\| \uhr[N] \|^2_{} + \frac{\dt}{4} \sum_{n=0}^{N-1} \|\uhr[n+1]\|^2_{\mathrm{SUPG}} \leq C_3 \left( A_0 + \dt	
		\sum_{n=0}^{N-1} \|f\|^2_{} \right) ,
	\end{equation}
	 with $C_3 = e^{(1 + 2 \nu) T}$.
\end{theorem}
\begin{proof}
	See appendix.
\end{proof}

As was the case in~\cite{jovo11}, this estimate induces a dependence of the SUPG parameter $\delta_K$ on $\dt$, which is not the ``expected behaviour'' (it has been experimentally observed in~\cite{jovo11} that the optimal SUPG parameter was independent of $\dt$). We therefore view this result as a sufficient set of conditions for ensuring norm-stability.

\section{Numerical experiments} \label{sec:num-exp}

This section highlights the stabilising properties of the SUPG-DLR framework through numerical experiments.

\subsection{Experiment 1 : rotating body problem}

We solve the classical ``rotating body'' problem, adapted from~\cite{jovo11, josc08} to the stochastic setting. In the determinstic setting, three ``shapes'' are advected on the domain $D = (0,1)^2$ by a counter-clockwise rotation field with unitary angular velocity centered in $(0.5, 0.5)$; furthermore, as the diffusion is small and there is no reaction term, after $t = 2\pi$ seconds, one expects to recover (almost) the initial condition. 

To adapt this example to our stochastic dynamical low-rank setting, we consider the following coefficients for~\eqref{eqn:adv-diff-reac}, 
\begin{align*}
	\varepsilon(\mathbf{y}) = 10^{y_1 - 16} && \bu(\mathbf{x}) = (0.5 - x_2, x_1 - 0.5)^{\top} && c = 0, 
\end{align*}
where $\mathbf{y} \in \Omega \coloneq [-1, 1]^3 $ is the parameter space, with measure $\displaystyle \mu(\mathrm{d}\mathbf{y}) = \frac{1}{8} \bigotimes_{i=1}^3 \lambda( \mathrm{d} y_i)$, where $\lambda$ is the Lebesgue measure.  


The ``shapes'' to be rotated are given by functions $U_0, U_1, U_2$, weighted by stochastic coefficients -- this expansion yields the initial condition. 
Let $r(x_1,x_2) = \sqrt{(x_1 - a)^2 + (x_2 - b)^2}/r_0$, where $r_0 = 0.15$. We then construct
\begin{equation}
	U_0(\mathbf{x}) = 
	\begin{cases}
		1 & \text{ if } r(x,y) \leq 1, \text{ and } |x_1 - a| \geq 0.025, x_2 \geq 0.85
		\\
		0 & \text{ else }
	\end{cases}, 
	\qquad 
	Y_0(\omega) = 1,  
\end{equation}
having used $(a,b) = (0.5, 0.75)$. This corresponds to a slotted cylinder. Next,  
\begin{equation}
	U_1(\mathbf{x}) = \frac{1}{4}(1 + \cos(\pi \min\{r(x_1,x_2), 1\})), \qquad \tilde{Y}_1(\mathbf{y}) = 2 y_2\cos(y_3)
\end{equation}
having set $(a, b) = (0.25, 0.5)$. This corresponds to a hump. Finally,
\begin{equation}
	U_2(\mathbf{x}) = 1 - \min\{ r(x_1,x_2), 1\}, \qquad \tilde{Y}_2(\mathbf{y}) = 30 y_3 (y_2)^3
\end{equation}
represents a cone, centered at $(a,b) = (0.5, 0.25)$. 

If the problem were a pure advection problem (hyperbolic), the true solution would be of rank $3$ (sum of a deterministic and two random zero-mean functions). The present problem is a singular perturbation of pure advection, and, as it displays a similar solution profile as its hyperbolic counterpart, we set $R=2$.

For the physical space, we used a uniform triangular mesh $\mathcal{T}_{h}$ of $128 \times 128$ and considered the $\mathbb{P}_1^C(\mathcal{T}_h)$-finite element space, yielding $N_h = 16641$ degrees of freedom. 
The stochastic space was discretised using $N_C = 7000$ independent and uniformly distributed samples from $\Omega$, yielding $\hat{\Omega} = \{{\omega}^{i}\}_{i=1}^{N_C}$. 

To ensure $U_0$ represents the mean of the random field at the discretised level, we subtract the sample mean of each stochastic mode, stack them in a matrix $[\tilde{Y}^{\star}_1, \tilde{Y}^{\star}_2]$, perform a weighted QR-factorisation to obtain a set of $\Lthm$-orthonormal modes, hereafter $Y_1$ and $Y_2$. The initial condition is then given by $u_0(x, \omega) = U_0(x) Y_0(\omega) + \sum_{i=1}^R U_i(x) Y_i(\omega)$. 
\begin{figure}[h]
\includegraphics[height=5cm]{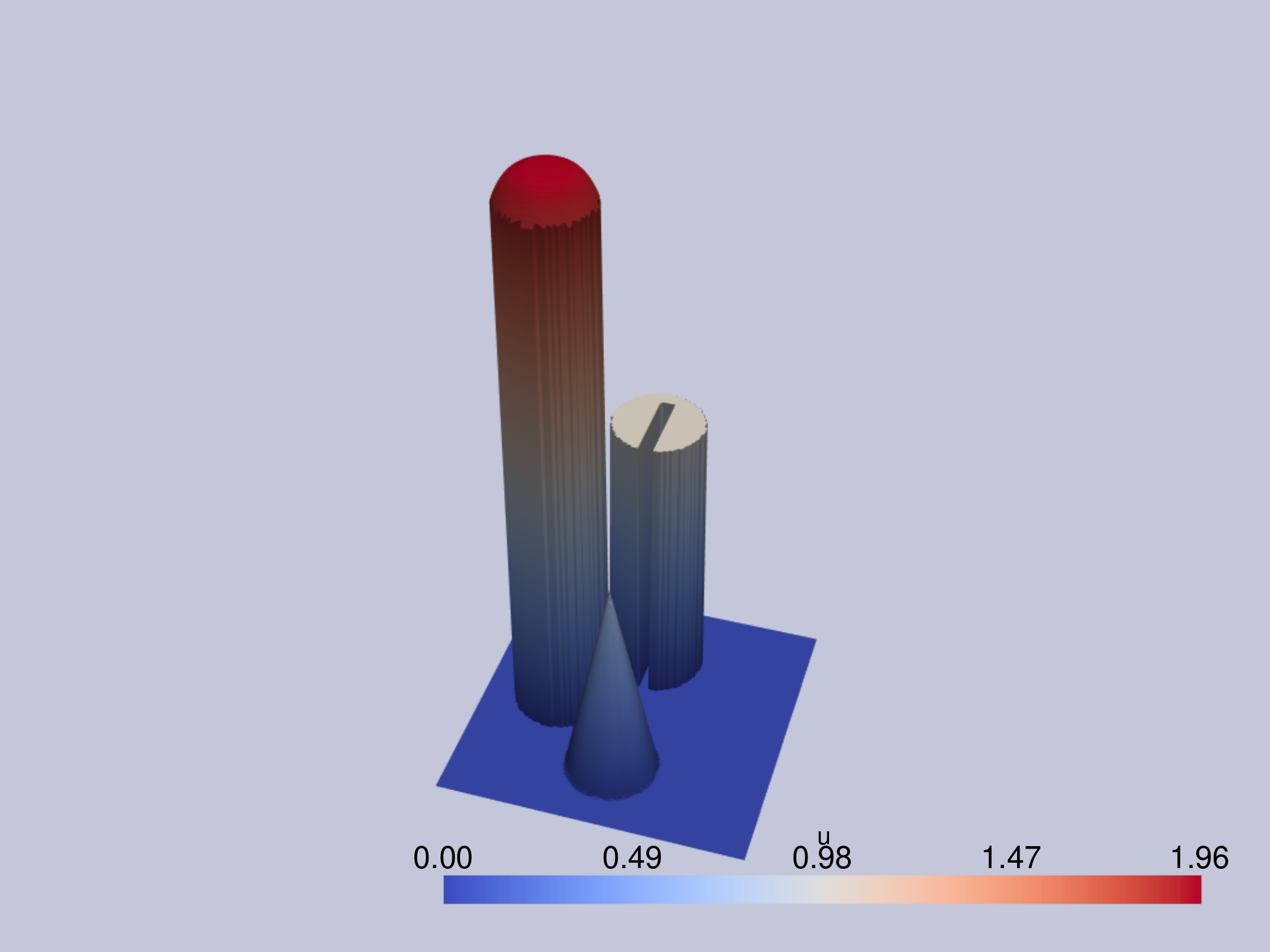}
\centering
\caption{Realisation at $t=0$ with parameters $\mathbf{y} = (0.05, -0.63, 0.67)$.}
\label{fig:initial-cond}
\end{figure}

We chose a time step $\dt = 2\pi / 70000 \approx 10^{-4}$. On the basis of the discussion in~\cite{jovo11}, we set the SUPG parameter to $\delta_K = \nicefrac{h_K}{4}$. We used the semi-implicit time-stepping scheme.  

%
%
%


In Figure~\ref{fig:real-dlr-vs-gal}, we display side-by-side one realisation at time $t = 2\pi$ of the numerical solution obtained via the standard DLR method and via the SUPG-stabilised DLR method. As can be observed, we recover a behaviour consistent with what has been discussed in~\cite{jovo11}: the SUPG method considerably damps the oscillations, although it doesn't cancel them entirely.

\begin{figure}
\centering
\begin{subfigure}{.45\textwidth}
  \centering
  \includegraphics[width=.9\linewidth]{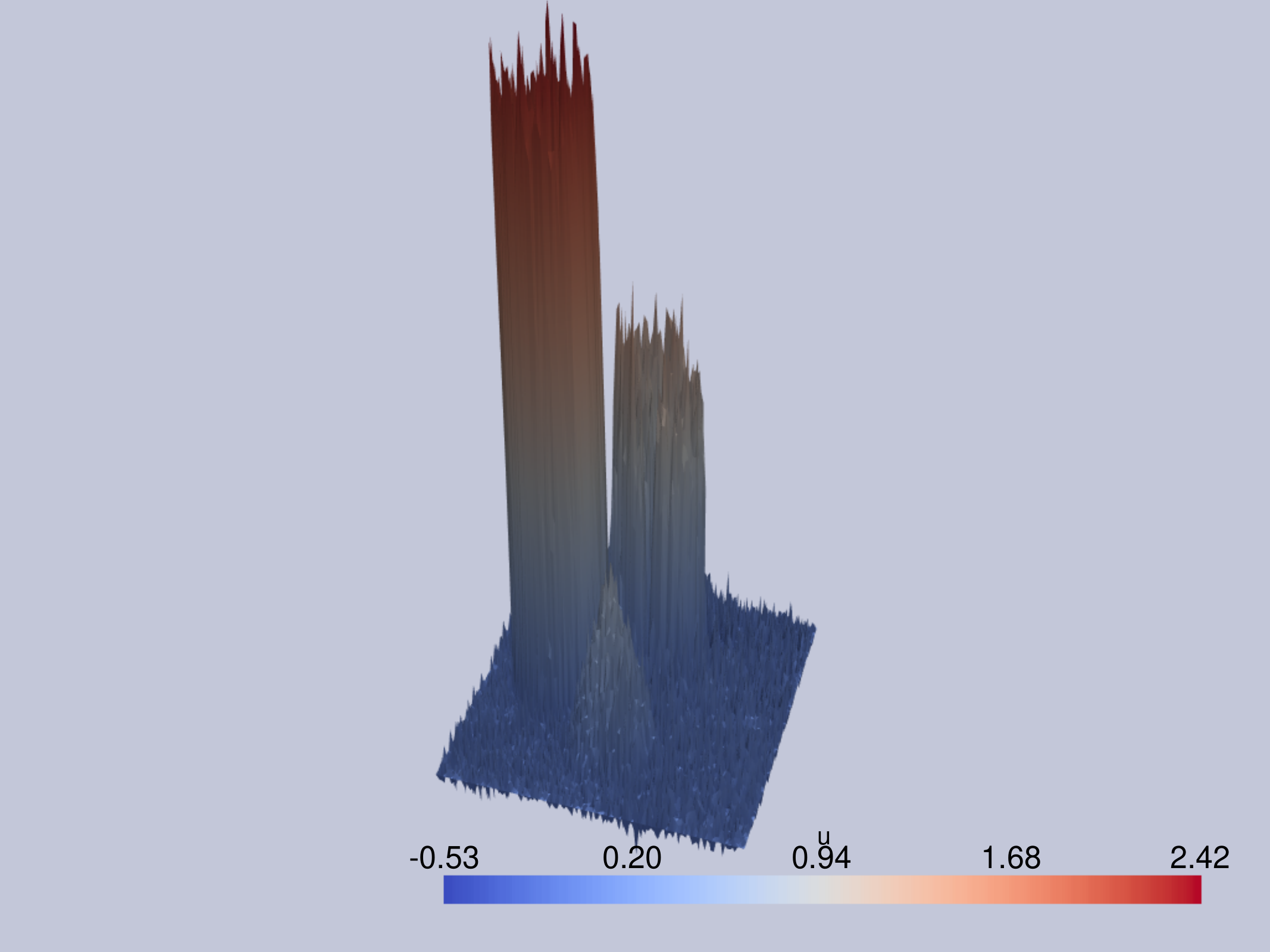}
  \caption{Standard DLR realisation}
  \label{fig:std-dlr-real}
\end{subfigure}%
\begin{subfigure}{.45\textwidth}
  \centering
  \includegraphics[width=.9\linewidth]{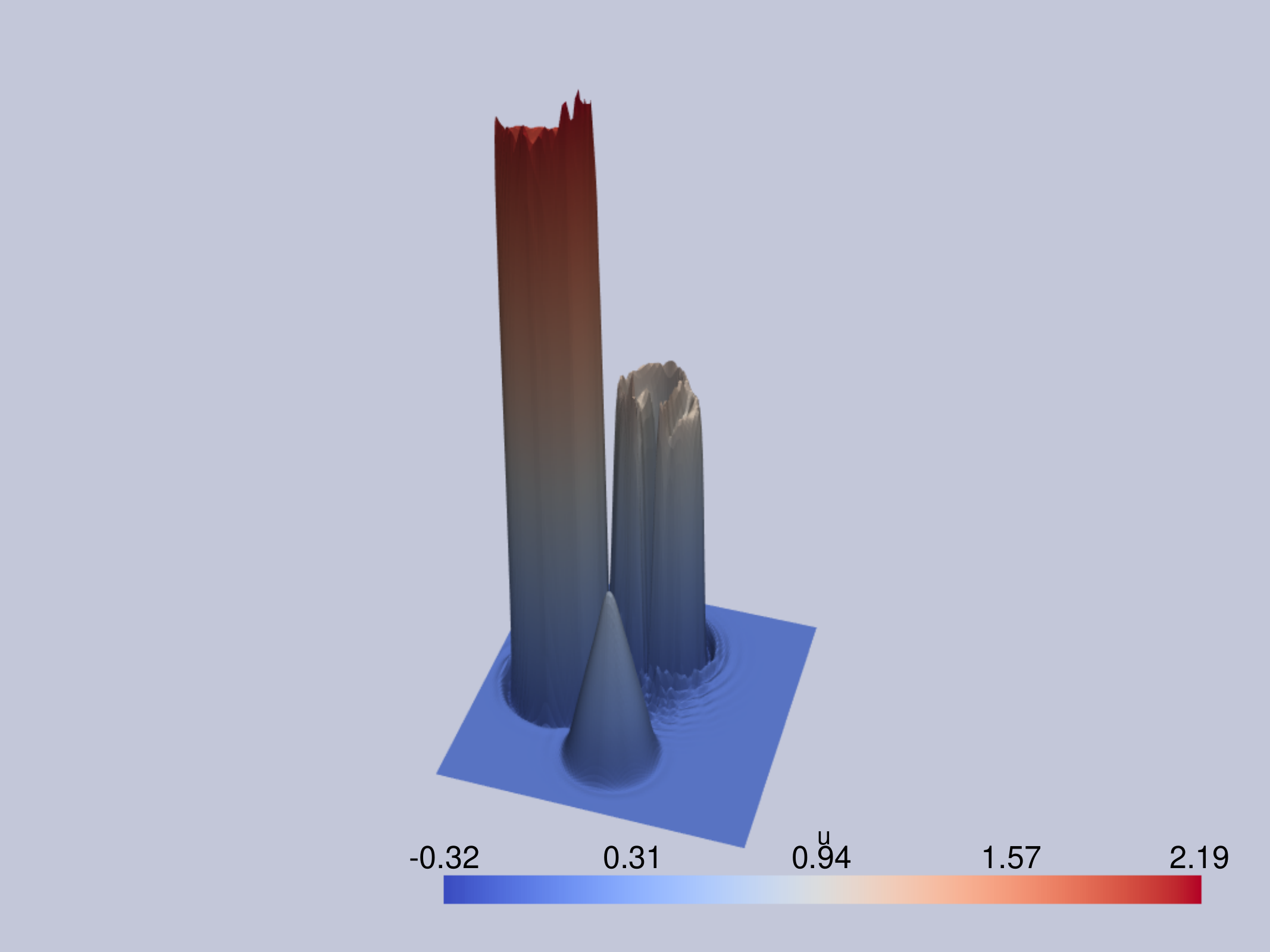}
  \caption{SUPG-DLR realisation}
  \label{fig:supg-dlr-real}
\end{subfigure}
\caption{Numerical solution for same parameters as in Figure~\ref{fig:initial-cond} obtained via standard DLR or SUPG-DLR at $t=2\pi$.}
\label{fig:real-dlr-vs-gal}
\end{figure}

\modt{We also investigate the accuracy of the SUPG-DLR compared to the standard DLR method and the impact of spurious oscillations.
To avoid a biased comparison between the Monte-Carlo solutions computed with SUPG stabilisation and the SUPG-DLR solutions, we compare the numerical solutions to a known proxy function sufficiently close to the true solution. 
Specifically, we choose $u_{\mathrm{proxy}}(t, x, \omega)$ to be the solution to the pure advection problem, i.e. with coefficients
\begin{align*}
	\varepsilon(\mathbf{y}) = 0 && \bu(\mathbf{x}) = (0.5 - x_2, x_1 - 0.5)^{\top} && c = 0, 
\end{align*}
which can be computed exactly by the method of lines. 
We verified that, over the time interval $[0,1]$, Monte-Carlo samples obtained with the SUPG method over a fine grid were significantly closer to the proxy solution (not shown in Figure~\ref{fig:xp1}); this makes it valid to use the proxy to compare the accuracy of $u_{\mathrm{DLR}}^{\mathrm{SUPG}}$ and $u_{\mathrm{DLR}}^{\mathrm{standard}}$. 
Furthermore, restricting ourselves to a small time interval allows to ensure that the dominating errors stem primarily from the physical discretisation through standard Galerkin or SUPG and not from the time discretisation. It is clear in Figure~\ref{fig:xp1} that the error of $u_{\mathrm{DLR}}^{\mathrm{standard}}$ grows quicker than that of $u_{\mathrm{DLR}}^{\mathrm{SUPG}}$, which we attribute to the greater amount of numerical oscillations present in $u_{\mathrm{DLR}}^{\mathrm{standard}}$. 
This is corroborated by Figure~\ref{fig:xp1-2}, in which we quantify the oscillations present in the realisations of both $u_{\mathrm{DLR}}^{\mathrm{SUPG}}$ and $u_{\mathrm{DLR}}^{\mathrm{standard}}$ using the Maximum Difference (MD) metric (called ``variance'' in~\cite{jovo11}), given by 
\[
  \mathrm{MD}(u(x)) = \max_{x \in D} u(x) - \min_{x \in D} u(x).
\] 
Given the quasi-pure transport nature of the problem, the theoretical behaviour is for the MD metric to remain constant for each realisation over time, with a value equal to that at the initial condition (since the initial condition gets principally advected over time). 
As is seen in Figure~\ref{fig:xp1-2}, the MD of $u_{\mathrm{DLR}}^{\mathrm{SUPG}}$ realisations deviate less from the MD of the initial condition, while also displaying a less oscillatory behaviour in comparison to their $u_{\mathrm{DLR}}^{\mathrm{standard}}$ counterparts.
}

\begin{figure}
\centering
\begin{subfigure}{.5\textwidth}
  \centering
	\includegraphics[scale=.45]{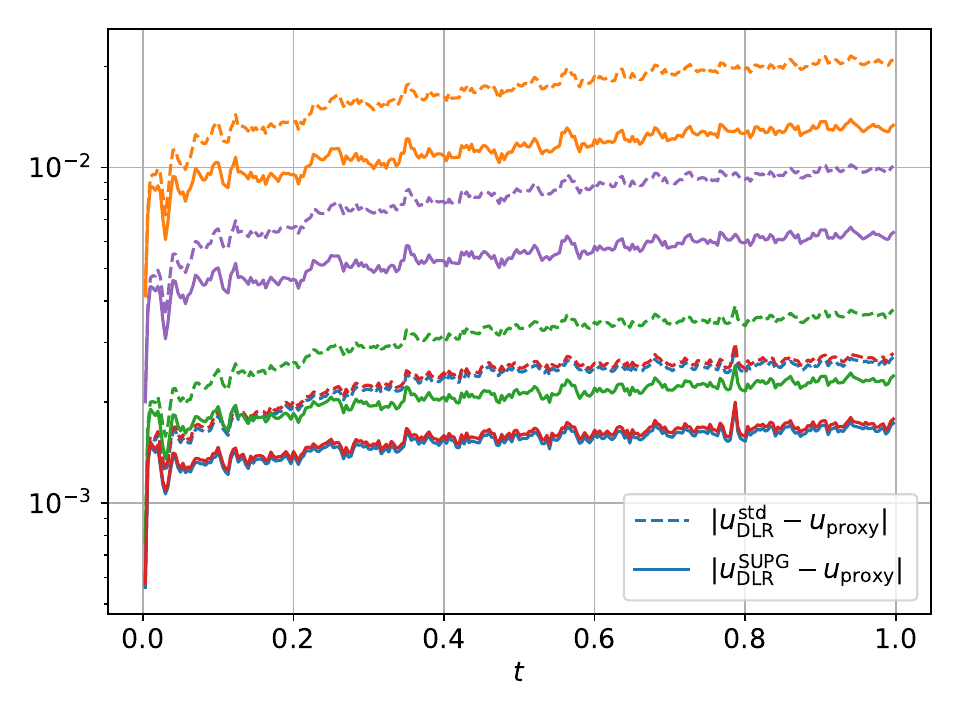}
  \caption{Accuracy w.r.t $u_{\mathrm{proxy}}$}
  \label{fig:xp1}
\end{subfigure}%
\begin{subfigure}{.5\textwidth}
  \centering
	\includegraphics[scale=.45]{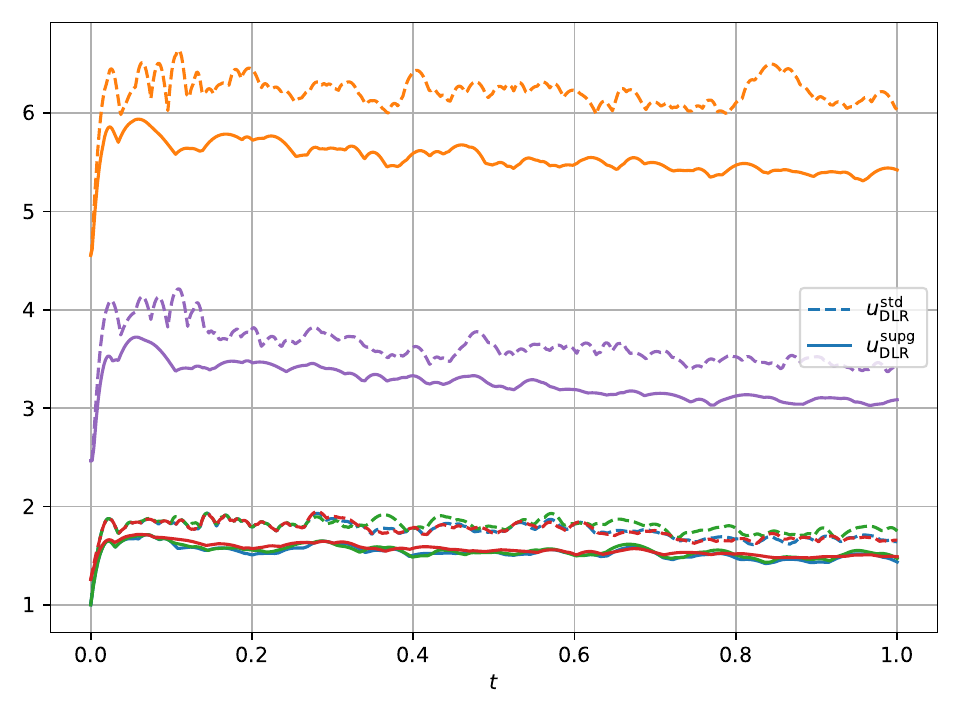}
  \caption{Oscillations (MD metric)}
  \label{fig:xp1-2}
\end{subfigure}
\caption{ (a) Comparison of $u_{\mathrm{DLR}}^{\mathrm{SUPG}}$ and $u_{\mathrm{DLR}}^{\mathrm{standard}}$ to proxy solution, each colour corresponds to a realisation. (b) Oscillations in realisations of $u_{\mathrm{DLR}}^{\mathrm{SUPG}}$ and $u_{\mathrm{DLR}}^{\mathrm{standard}}$. }
\label{fig:test}
\end{figure}


\subsection{Experiment 2 : boundary layers \& mildly stochastic advection}

We consider a variation of the experiment proposed in~\cite{paro14}.
On the physical domain $D = (0,1)^2$ and on the time interval $[0,T]$ with $T=1.2$ (with $\dt = \nicefrac{T}{50}$), we solve the transient random advection-diffusion-reaction~\eqref{eqn:adv-diff-reac}, with coefficients
\begin{align*}
	\varepsilon(\mathbf{y}) = \frac{1}{y_1}, && \bu(\mathbf{x}, \mathbf{y}) = 
	\begin{pmatrix}
		1 \\
		1
	\end{pmatrix}
	+ (y_2 - k(y_2)) 
	\begin{pmatrix}
		x_2 \\
		x_1
	\end{pmatrix}
	, && c = 0, 
\end{align*}
where $\mathbf{y} = (y_1, y_2, y_3, y_4) \in [5000,6000] \times  [-1,1]^3 \eqcolon \Omega$ (the parameter space), and $k(y_2)$ is a constant that will be defined below.
The physical domain is discretised using a uniform triangular mesh $\mathcal{T}_{\!h}$ of $50 \times 50$, on which we consider the $\mathbb{P}_{1}^C(\mathcal{T}_{\!h})$ Finite Element space of size $N_h = 2601$. 
The stochastic space is discretised by collocating each parameter interval with $N =10$ equispaced points, i.e., $\hat{\Omega} = \mathcal{S}_N^1 \times \ldots \times \mathcal{S}^4_N$, where $\mathcal{S}^i_N = \{a_i + j\frac{b_i - a_i}{(N-1)}, j = 0, \ldots, N-1\}$, hence $N_C = 10^4$. The discrete measure is given by $\hat{\mu} = N_C^{-1} \sum_{i=1}^{N_C} \delta_{\mathbf{y}^{(i)}}$. 

The initial condition is given by a suitable truncation of the zero-mean random function
\begin{equation}
	\hat{u}(\mathbf{x}, \mathbf{y})^{\star} = 5 \sin(2 \pi x_1) \sin(2 \pi x_2) 
	(e^{
		\cos(y_3 x_1 + y_4 x_2)
	})^{\star}.
\end{equation}
We perform a generalised SVD~\cite{ab07} to obtain $\mathrm{GSVD}(\hat{u}^{\star}) = \sum_{i=1}^{\min(N_h, N_C)} U_i Y_i$ and truncate at $R = 34$, as this corresponds to a truncation error of order $10^{-5}$, smaller than the error incurred by the Finite Element method. 
Setting $U_0 = 0$ (and recalling that $Y_0$ = 1), the initial condition is given by $u_0(\mathbf{x}, \mathbf{y}) = \sum_{i=0}^R U_i(\mathbf{x}) Y_i (\mathbf{y})$. 

For $\partial D_1$ and $\partial D_2$ as specified in Figure~\ref{fig:domain-boundary}, the boundary conditions are given by $u(t, \mathbf{x}, \mathbf{y}) = 1$ for $\mathbf{x} \in \partial D_1$ and $0$ when $\mathbf{x} \in \partial D_2$.
These boundary conditions are not homogeneous, but remain deterministic.
Following the approach proposed in~\cite{muno18}, the non-homogeneous boundary conditions can be imposed on $U_0$ while the other physical modes $\{U_i\}_{i=1}^R$ are supplemented with homogeneous boundary conditions.

By the imposed boundary conditions, there is a formation of a boundary layer, making the numerical simulation challenging. 
Furthermore, since the advection is stochastic, it is not directly amenable to the generalised PG-DLR framework described in this work. 
As suggested earlier, we use the mean advection field $\mathbb{E}_{\hat{\mu}}[\bu]$ to stabilise the DLR system. 
This is justified by the fact that the realisations follow the same overall direction as the mean field, and the SUPG is expected to alleviate oscillations arising from sharp gradients parallel to the advection field.
To ensure that $\mathbb{E}_{\hat{\mu}}[\bu] = (1,1)^{\top}$, we choose $k(y_2) = \mathbb{E}_{\hat{\mu}}[y_2]$ in the problem, which causes the second part of the advection field to be naturally zero-mean in the chosen stochastic discretisation. 

Figure~\ref{fig:reals-time} shows the evolution of a realisation obtained via the SUPG-DLR method at different times; we do not show the plots obtained via the standard DLR method as, for the same parameters, \modt{the resulting numerical solutions were extremely oscillatory and unstable}. 

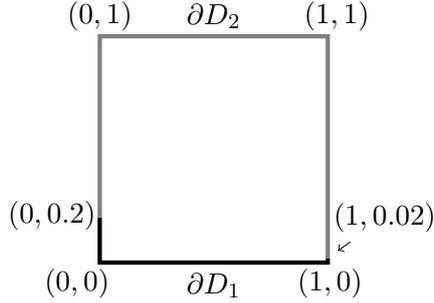
\begin{figure}
	\centering
	\begin{tikzpicture}[scale=3]
\draw[black, ultra thick] (0,.2) -- (0,0) -- (1,0) -- (1, .02);
\draw[gray, ultra thick] (0,.2) -- (0,1) -- (1,1) -- (1, .02);
\node (1) at (-.1,-.09) {$(0,0)$};
\node (2) at (1.04,1.09) {$(1,1)$};
\node (3) at (1.01,-.09) {$(1,0)$};
\node (4) at (0,1.09) {$(0,1)$};
\node (5) at (-0.21, .21) {$(0,0.2)$};
\node (6) at (1.25, .2) {$(1,0.02)$};
\node (7) at (1., 0.02) {};
\draw [->] (6) -- (7) ; 
\node (8) at (.5, -0.1) {$\partial D_1$};
\node (9) at (.5, 1.09) {$\partial D_2$};
\end{tikzpicture}
	\caption{Boundary of the domain, separated in $\partial D_1$ (black) and $\partial D_2$ (grey).}
	\label{fig:domain-boundary}
\end{figure}

\begin{figure}
\centering
\begin{subfigure}{.33\textwidth}
  \centering
  \includegraphics[width=.9\linewidth]{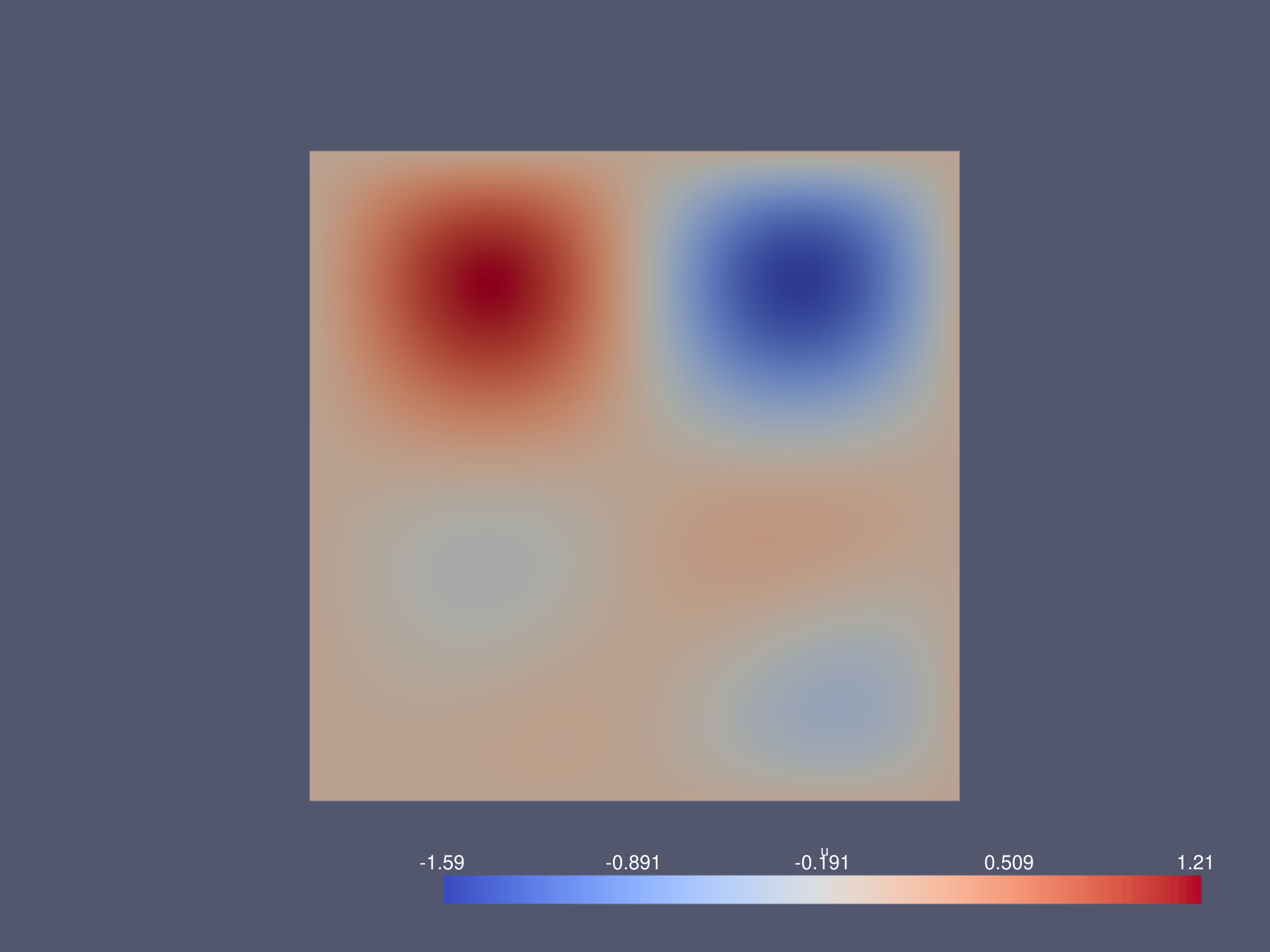}
  \caption{$t=0$}
  \label{fig:real-t1}
\end{subfigure}%
\begin{subfigure}{.33\textwidth}
  \centering
  \includegraphics[width=.9\linewidth]{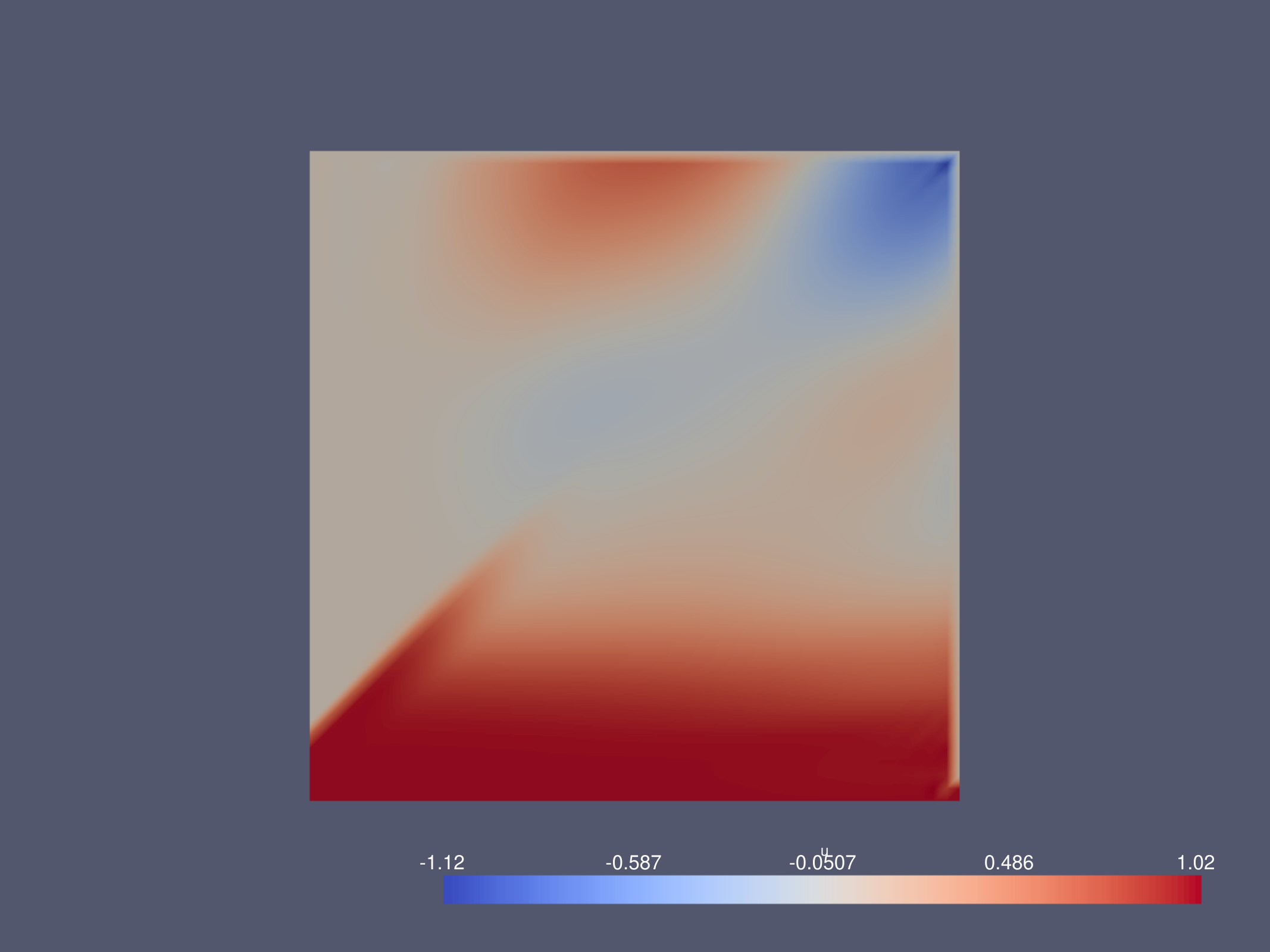}
  \caption{$t = 0.24$}
  \label{fig:real-t2}
\end{subfigure}
\begin{subfigure}{.33\textwidth}
  \centering
  \includegraphics[width=.9\linewidth]{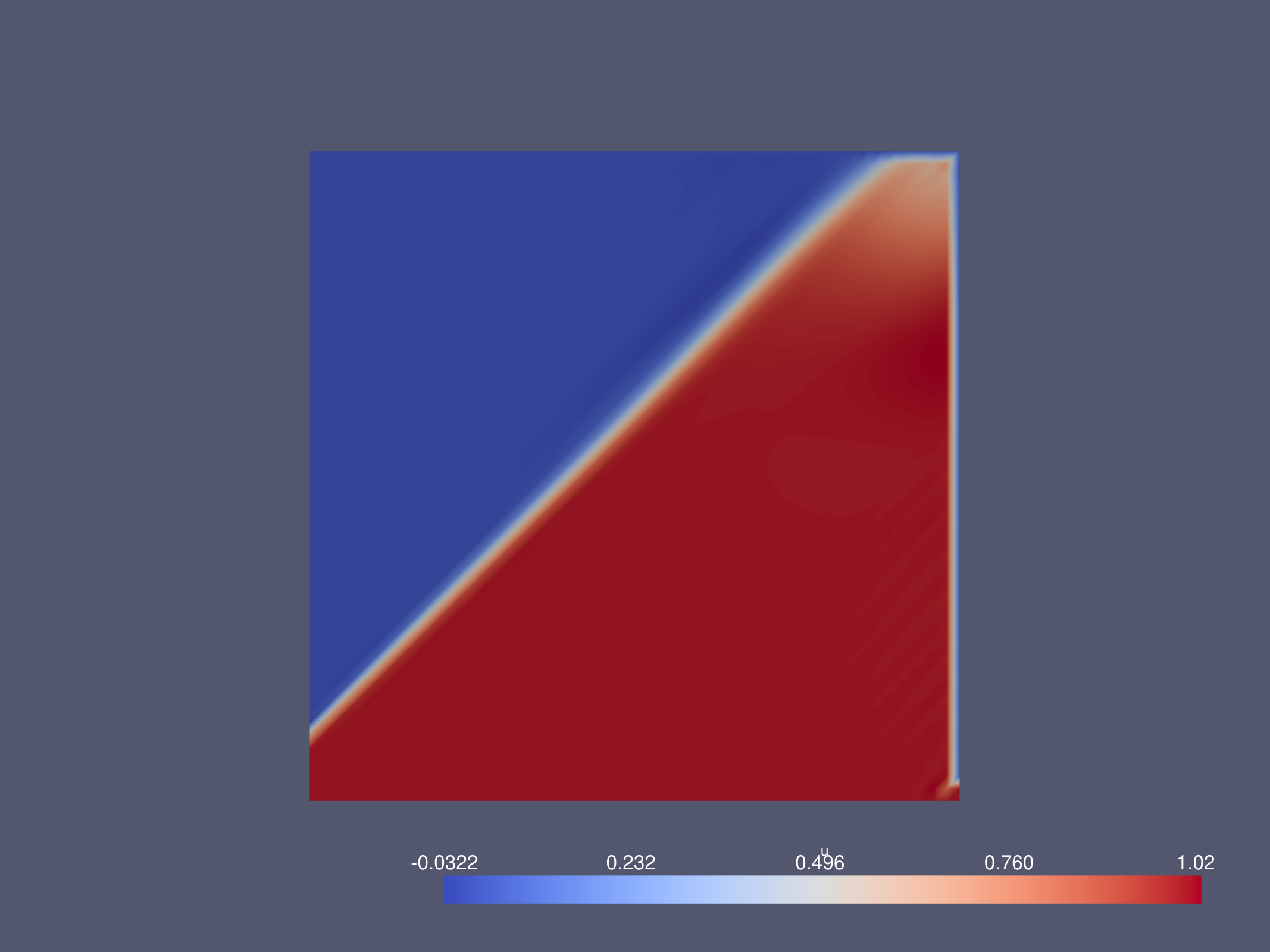}
  \caption{$t = 1.2$}
  \label{fig:real-t3}
\end{subfigure}
\caption{Evolution of parametrised solution for $\mathbf{y} = (1.6 \cdot 10^{-4} , 1, -0.33, -0.77)$.}
\label{fig:reals-time}
\end{figure}



\bibliographystyle{siam} 
\bibliography{sdraft/refs} 

\section{Appendix}

\begin{proof}[Proof of Lemma~\ref{lem:coerasupg}] For notational convenience, denote $u_{\hhmu}$ by $u$. We have
	\begin{multline} \label{eqn:mid-dev1}
	\aSUPG(u,u) = (\varepsilon \nabla u, \nabla u) + (\bu \cdot \nabla u, u) + (cu, u) + \sum_{K \in \Th} \delta_K (- \varepsilon \Delta u + \bu \cdot \nabla u + cu, \bu \cdot \nabla u)_{K,\Lthm} 
	\\
	\geq \hat{\varepsilon} \|\nabla u\|^2 + ((c - \frac{1}{2} \mathrm{div} \, \bu + \nu)u), u) - \nu\|u\|^2 + \sum_{K \in \Th} \delta_K \|\bu \cdot \nabla u\|^2_{K,\Lthm} 
	\\ 
	+  \sum_{K \in \Th} \delta_K (- \varepsilon \Delta u + cu, \bu \cdot \nabla u)_{K,\Lthm}. 
	\end{multline}
	The last term is bounded by
	\begin{multline*}
		|\sum_{K \in \Th} \delta_K (c u, \bu \cdot \nabla u)_{K,\Lthm}| 
		\leq 
		\sum_{K \in \Th} \delta_K \trnorm{c}^{\nicefrac{1}{2}}_K \| |c|^{\nicefrac{1}{2}} u\|_{K,\Lthm}
		\|\bu \cdot \nabla u\|_{K,\Lthm} 
		\\
		\leq \sum_{K \in \Th} \delta_K
		\left(
			\trnorm{c}_{K} \| |c|^{\nicefrac{1}{2}}u\|_{K,\Lthm}^2 + \frac{\|\bu \cdot \nabla u\|^2_{K,\Lthm}}{4}
		\right) 
		\leq  
		\frac{\| |c|^{\nicefrac{1}{2}}u\|^2}{2} + \frac{1}{4} \sum_{K \in \Th} \delta_K \|\bu \cdot \nabla u\|^2_{K,\Lthm}, 
	\end{multline*}
	having used~\eqref{eqn:deltaK-coerc}. Proceeding in the same fashion, 
	\begin{align*}
		|\sum_{K \in \Th} \delta_K (- \varepsilon \Delta u, &\bu \cdot \nabla u)_{K,\Lthm}| 
		\leq  
		\sum_{K \in \Th} \delta_K \|\varepsilon \Delta u\|_{K,\Lthm} \|\bu \cdot \nabla u\|_{K,\Lthm} 
		\\
																												&\leq 
		\sum_{K \in \Th} \delta_K C_E \hat{\varepsilon} \sqrt{d} \frac{C_I}{h_K}\| \nabla u\|_{K,\Lthm} \|\bu \cdot \nabla u\|_{K,\Lthm}
		\\ 
																												&\leq
		\sum_{K \in \Th} \delta_K \left(
			\frac{ d C_E^2 C_I^2}{h^2_K} \hat{\varepsilon}^2 \|\nabla u\|^2_{K,\Lthm} + \frac{\|\bu \cdot \nabla u\|^2_{K,\Lthm}}{4}
		\right)
		\\ 
																												&\leq \frac{\hat{\varepsilon}}{2} \|\nabla u \|^2 
		+ 
		\frac{1}{4} \sum_{K \in \Th} \delta_K \|\bu \cdot \nabla u\|^2_{K,\Lthm}.
	\end{align*}
	Using these bounds in~\eqref{eqn:mid-dev1} and noticing that
	\begin{equation}
		(	(c - \frac{1}{2} \mathrm{div}\,\bu + \nu) u, u) - \frac{1}{2}\| |c|^{\nicefrac{1}{2}} u\|^2 = (	(c - \frac{|c|}{2} - \frac{1}{2} \mathrm{div}\,\bu + \nu) u, u) = \|\mu^{\nicefrac{1}{2}}  u \|^2_{},
	\end{equation}
	yields the result. 
\end{proof}

\begin{proof}[Proof of Theorem~\ref{th:si-stab}]
	Firstly, as the bounds on $\delta_K$ in~\eqref{eqn:si-dK-bound} are smaller than~
	\eqref{eqn:deltaK-coerc}, Lemma~\ref{lem:coerasupg} still holds, and the same steps as the proof in Theorem~\ref{th:im-stab} can be applied. We test~\eqref{eqn:varf-adr-disc-si} against $\uhr[n+1]$, and add and subtract the quantity
	\begin{equation}
		(\varepsilon^{\star} \nabla \uhr[n+1], \uhr[n+1]) + ( c^{\star} \uhr[n+1], \uhr[n+1]) + \sum_{K \in \Th} \delta_K (- \varepsilon^{\star} \Delta \uhr[n+1] + c^{\star} \uhr[n+1], \bu \cdot \nabla \uhr[n+1])_K.
	\end{equation}
	We thus recover the variational formulation of the implicit Euler with some extra terms: 
	\begin{multline} \label{eqn:varfsi}
		\frac{1}{2\dt}(\|\uhr[n+1] \|_{}^2 -  \| \uhr[n] \|_{}^2 +  \| \uhr[n+1]- \uhr[n] \|_{}^2) + \aSUPG(\uhr[n+1], \uhr[n+1])
		\\ 
		= -  \sum_{K \in \Th} \frac{\delta_K}{\dt} (\uhr[n+1] - \uhr[n], \bu \cdot \nabla \uhr[n+1]) +  (f, \uhr[n+1]) + \sum_{K \in \Th} \delta_K (f, \bu \cdot \nabla \uhr[n+1]) 
		\\ + (\varepsilon^{\star} \nabla (\uhr[n+1] - \uhr[n]), \nabla \uhr[n+1]) + (c^{\star} (\uhr[n+1] - \uhr[n]), \uhr[n+1]) 
		\\ 
		+ \sum_{K \in \Th} \delta_K (- \varepsilon^{\star} \Delta \uhr[n+1] + c^{\star} (\uhr[n+1] - \uhr[n]), \bu \cdot \nabla \uhr[n+1])_K.
	\end{multline}
	Upon bounding the extra terms in a suitable fashion, we can apply again the same steps as the proof for theorem~\ref{th:im-stab} (albeit with modified constants). The following terms are bounded as follows:  
	\begin{align}
		|(\varepsilon^{\star} \nabla (\uhr[n+1] - \uhr[n]), \nabla \uhr[n+1])| 
		&\leq 
		\frac{3}{64} \hat{\varepsilon} \|\nabla \uhr[n+1]\|^2_{} + \frac{1}{64} C_1 \hat{\varepsilon} \|\nabla \uhr[n]\|^2_{}, \label{eqn:eps-star-bound}
		\\
		|(c^{\star} (\uhr[n+1] - \uhr[n]), \uhr[n+1])| 
		&\leq 
		\frac{3}{64}  \|\mu^{\nicefrac{1}{2}} \uhr[n+1]\|^2_{} + \frac{1}{64} \| \mu^{\nicefrac{1}{2}} \uhr[n]\|^2_{}, \label{eqn:cstar-bound}
	\end{align}
	and
	\begin{multline}	 \label{eqn:cstar-delta-bound}
		\sum_{K \in \Th} \delta_K |(c^{\star} (\uhr[n+1] - \uhr[n]), \bu \cdot \nabla \uhr[n+1])_K| 
		\\
		\leq \sum_{K \in \Th} \frac{1}{4} \delta_K \trnorm{c^{\star}}_K (\| \mu^{\nicefrac{1}{2}} \uhr[n+1] \|_{K,\Lthm}^{2} +  \| \mu^{\nicefrac{1}{2}} \uhr[n]\|_{K,\Lthm}^2 ) +\frac{\delta_K}{16}\|\bu \cdot \nabla  \uhr[n+1] \|^2_{K, \Lthm}
		\\
		\leq \frac{1}{32} (\| \mu^{\nicefrac{1}{2}} \uhr[n+1] \|_{}^{2} +  \| \mu^{\nicefrac{1}{2}} \uhr[n]\|_{}^2 ) + \sum_{K \in \Th} \frac{\delta_K}{16}\|\bu \cdot \nabla  \uhr[n+1] \|^2_{K, \Lthm},
	\end{multline}
	and similarly 
	\begin{multline} \label{eqn:eps-star-delta-bound}
		\sum_{K \in \Th} \delta_K (- \varepsilon^{\star} \Delta (\uhr[n+1]  - \uhr[n]), \bu \cdot \nabla \uhr[n+1])_{K,\Lthm} 
		\\
		\leq 
		\sum_{K \in \Th} \frac{\delta_K}{4} \frac{\hat{\varepsilon}^2 C_I^2 d}{h^2_K}  (  \|\nabla \uhr[n+1]\|^2 + \|\nabla \uhr[n]\|^2 )
		+ \frac{\delta_K}{16}\|\bu \cdot \nabla  \uhr[n+1] \|^2_{K, \Lthm}
		\\
		\leq \frac{\hat{\varepsilon}}{32} (  \|\nabla \uhr[n+1]\|^2 + \|\nabla \uhr[n]\|^2 ) + \sum_{K\in \Th} \frac{\delta_K}{16} \|\bu \cdot \nabla \uhr[n+1]\|^2_{K,\Lthm}.
	\end{multline}
Summing over $n=0,\ldots,N-1$, we can collect terms of same $n$ to get the upper bound 
\begin{equation*}
	\frac{1}{8} \hat{\varepsilon} \|\nabla \uhr[n+1]\|^2_{} + \frac{1}{8}\|\mu^{\nicefrac{1}{2}} \uhr[n+1]\|^2_{} + \frac{1}{8} \sum_{K \in \Th} \|\bu \cdot \nabla \uhr[n+1]\|^2_{K,\Lthm}
\end{equation*}
for any $n \in 0, \ldots, N-1$. In the proof of Theorem~\ref{th:im-stab}, this term can be absorbed by the remaining coercive term from~\eqref{eqn:coer-bound1} or~\eqref{eqn:coer-bound2}, yielding
\begin{multline*}
	C_1 \|\uhr[n+1]\|^2_{\mathrm{SUPG}} -  
	\frac{1}{8} \hat{\varepsilon} \|\nabla \uhr[n+1]\|^2_{} - \frac{1}{8}\|\mu^{\nicefrac{1}{2}} \uhr[n+1]\|^2_{} - \frac{1}{8} \sum_{K \in \Th} \|\bu \cdot \nabla \uhr[n+1]\|^2_{K,\Lthm} 
	\\
	\geq C_2 \|\uhr[n+1]\|^2_{\mathrm{SUPG}},
\end{multline*}
where $C_1$ and $C_2$ depend on the case. In cases \textit{(i)} and \textit{(iii)}, $C_1 = \frac{1}{4}$ and hence $C_2 = \frac{1}{8}$, and in case \textit{(ii)} $C_1 = \frac{3}{8}$ and $C_2 = \frac{1}{4}$. The only additional terms that must be accounted for are the terms corresponding to $n = 0$, which are then found on the right-hand-side. Continuing the proof exactly as in Theorem~\ref{th:im-stab} yields the result.
\end{proof}



\end{document}